\documentclass[12pt]{amsart}
\usepackage[T1]{fontenc}
\usepackage[latin1]{inputenc}
\usepackage{typearea}
\usepackage{geometry}
\usepackage{ulem}
\usepackage{amsmath}
\usepackage{amssymb}
\usepackage{latexsym}
\usepackage{enumerate}
\usepackage{amsthm}
\usepackage[all]{xy}
\usepackage{hhline}
\usepackage{epsf} 
\usepackage{cite}
\usepackage{verbatim}
\usepackage{mathtools}
\usepackage{verbatim}
\usepackage{dsfont}

\renewcommand{\1}{\mathds{1}}
\allowdisplaybreaks

\usepackage{color}

\newcommand{\V}{\mathbb{V}}

\newtheorem{theorem}{{\sc Theorem}}[section]
\newtheorem{corollary}[theorem]{{\sc Corollary}}

\newtheorem{lemma}[theorem]{{\sc Lemma}}
\newtheorem{proposition}[theorem]{{\sc Proposition}}

\theoremstyle{remark}
\newtheorem{remark}[theorem]{{\sc Remark}}

\theoremstyle{definition}
\newtheorem{definition}[theorem]{\sc definition}

\newcommand{\R}{\mathbb{R} }
\newcommand{\N}{\mathbb{N} }

\newcommand{\B}{\mathcal{B}}
\newcommand{\F}{\mathcal{F}}
\newcommand{\G}{\mathcal{G}}

\newcommand{\D}{\mathcal{D}}
\newcommand{\W}{\mathcal{W}}

\newcommand{\Sm}{\mathbb{S}_m}

\newcommand{\Sp}{\mathbb{S}_p}

\newcommand{\Tr}{\textnormal{Tr}}
\newcommand{\Prob}{\mathbb{P}}
\newcommand{\E}{\mathbb{E}}

\newcommand{\xbf}{\mathbf{x}}
\newcommand{\ybf}{\mathbf{y}}
\newcommand{\tbf}{\mathbf{t}}
\newcommand{\sbf}{\mathbf{s}}
\newcommand{\wbf}{\mathbf{w}}

\providecommand{\abs}[1]{\lvert #1\rvert}
\providecommand{\babs}[1]{\bigl\lvert #1\bigr\rvert}

\providecommand{\norm}[1]{\lVert #1\rVert}
\providecommand{\fnorm}[1]{\lVert #1\rVert_\infty}

\providecommand{\Opnorm}[1]{\lVert #1\rVert_{\op}}
\providecommand{\Enorm}[1]{\lVert #1\rVert_2}
\providecommand{\HSnorm}[1]{\lVert #1\rVert_{\HS}}
\DeclareMathOperator{\diag}{diag}
\DeclareMathOperator{\Var}{Var}

\DeclareMathOperator{\Cov}{Cov}

\DeclareMathOperator{\op}{op}
\DeclareMathOperator{\Hess}{Hess}
\DeclareMathOperator{\HS}{H.S.}

\DeclarePairedDelimiter{\ceil}{\lceil}{\rceil}
\DeclarePairedDelimiter{\floor}{\lfloor}{\rfloor}

\renewcommand{\phi}{\varphi}
\renewcommand{\epsilon}{\varepsilon}

\renewcommand{\rho}{\varrho}

\begin{document}

\title[CLTs for symmetric $U$-statistics]{Quantitative CLTs for symmetric $U$-statistics \\ using contractions}
\author{Christian D\"obler \and Giovanni Peccati}
\thanks{\noindent Universit\'{e} du Luxembourg, Unit\'{e} de Recherche en Math\'{e}matiques \\
E-mails: christian.doebler@uni.lu, giovanni.peccati@uni.lu
}
\begin{abstract}  {We consider sequences of symmetric $U$-statistics, not necessarily Hoeffding-degenerate, both in a one- and multi-dimensional setting, and prove quantitative central limit theorems (CLTs) based on the use of {\it contraction operators}. Our results represent an explicit counterpart to analogous criteria that are available for sequences of random variables living on the Gaussian, Poisson or Rademacher chaoses, and are perfectly tailored for geometric applications. As a demonstration of this fact, we develop explicit bounds for subgraph counting in generalised random graphs on Euclidean spaces; special attention is devoted to the so-called `dense parameter regime' for uniformly distributed points, for which we deduce CLTs that are new even in their qualitative statement, and that substantially extend classical findings by Jammalamadaka and Janson (1986) and Bhattacharaya and Ghosh (1992).}

\smallskip

\end{abstract}

\maketitle

\section{Motivation and Overview}\label{s:intro0}

\subsection{ Introduction} In the recent reference \cite{DP16}, we have provided a multidimensional and quantitative version of a seminal result by de Jong \cite{deJo89, deJo90}, roughly stating that, {\it if ${\bf F} = \{ F_n : n\geq 1\}$ is a normalized sequence of random variables having the form of degenerate, not necessarily symmetric $U$-statistics of a fixed order, and ${\bf F}$ enjoys an appropriate Lindeberg property, then a sufficient condition for $F_n$ to verify a central limit theorem (CLT) (as $n\to\infty$) is that $\E[F_n^4] \to 3 $.} Observe that $3 = \E[N^4]$, where $N\sim \mathcal{N}(0,1)$ is a standard Gaussian random variable. 

\smallskip

 The aim of this paper is to develop some remarkable applications and extensions of the main results of \cite{DP16, deJo89, deJo90} to the case of {\it symmetric} and degenerate $U$-statistics, in a possibly multidimensional setting.
 {By symmetric we mean here that the corresponding kernel does not depend on the choice of the subset of {the} random input, although it might well vary with the sample size $n$.} In particular, our main aim is to establish a collection of quantitative one-  and multi- dimensional CLTs (see Theorem \ref{1dimbound} and Theorem \ref{mdimbound} below), with explicit bounds expressed in terms of {\it contraction operators} -- see Section \ref{intro} below as well as \cite[Section 6]{Lastsv} for definitions. Our tools will involve new multiplication formulae for $U$-statistics (see Proposition \ref{pform}), as well as new estimates on contraction operators (see Lemma \ref{contrlemma}), that seem to have an independent interest. 

\smallskip

Although the previously quoted results only involve degenerate $U$-statistics, we will show in Section 5 that they can be naturally generalized to the case of arbitrary symmetric $U$-statistics, {by exploiting the explicit form of}{ their Hoeffding decomposition, {together with} our multivariate results.} As discussed in great detail in the two monographs \cite{NouPecbook, PecRei16}, as well as in the papers \cite{LRP1, LRP2, NPR10b}, contraction operators play a fundamental role in CLTs involving random variables belonging to the Wiener chaos of a Gaussian field, of a Poisson measure or of a Rademacher sequence. To the best of our knowledge, our contributions represents the first systematic use of contraction operators in the framework of general symmetric $U$-statistics.

\smallskip

As the discussion in \cite{BPsv} and the references therein largely demonstrates, the use of contraction operator is well-adapted for dealing with geometric applications, involving e.g. additive functionals of random geometric graphs, like the total length, or subgraph counting statistics. In the last section of the present paper, we will apply our results to edge-counting statistics of geometric random graphs, belonging to the family of geometric structures studied in \cite{Penrose-book}, thus substantially generalising some estimates from \cite{LRP2}, as well as from the classical references \cite{BhaGo92, JJ}.

\smallskip

\subsection{ Comments on previous literature}  Due to the generality of our results, the present work is related to most articles dealing with the asymptotic normality of symmetric $U$-statistics, like the classical paper \cite{Hoeffding} about non-degenerate $U$-statistics given by fixed kernels, as well as the more recent papers \cite{JJ, Hall84, BhaGo92, Weber, RiRo97}, in which the considered kernels might well depend on the sample size $n$. We would like to point out explicitly that, like this work, also the papers \cite{JJ, BhaGo92} prove asymptotic normality of one-dimensional $U$-statistics that do not necessarily have a dominant Hoeffding component via a multivariate CLT for the vector of Hoeffding components. Our method can be seen as a quantitative counterpart to such an approach. Moreover, whereas the references \cite{JJ, Hall84, BhaGo92, Weber} provide in general non-equivalent and {very technical} conditions for asymptotic normality, our statements {will only involve simple analytic quantities}, merely depending on norms of contractions of the kernels. We believe that, as in the Poisson situation \cite{LRP1, LRP2}, such conditions are most suitable {for a large array of} possible applications --- plausibly much wider than the set of examples discussed in the present paper. Finally, although for symmetric $U$-statistics the results of \cite{DP16} imply asymptotic normality whenever each Hoeffding component satisfies a fourth moment condition, these moment conditions are generally quite hard to check in practice. This remark applies even more so, when the $U$-statistic is 
nondegenerate so that one would have to deal with the complicated expressions for the kernels appearing in the Hoeffding decomposition.

\smallskip

As in \cite{DP16}, our results rely on Stein's method of exchangeable pairs \cite{St86}. Other articles which have proved (quantitative) CLTs for $U$-statistics via this approach include \cite{ReiRoe10, RiRo97}. However, since \cite{ReiRoe10} only deals with 
non-degenrate kernels that do not depend on $n$, the overlap with the present paper seems marginal. In \cite{RiRo97}, {the class of} so-called {\it weighted $U$-statistics} is considered, and CLTs are obtained for non-degenrate kernels of arbitrary order as well as 
for degenerate kernels of order $2$. In the latter case, and when all weights are set to $1$, our bound in Theorem \ref{1dimbound} not only improves on \cite[Theorem 1.4]{RiRo97} with respect to the rate of convergence but also deals with degenrate 
kernels of arbitrary orders.

\smallskip

We eventually observe that an alternate approach for obtaining the main results of the present paper (in particular, the general bounds of Section 5) could be based, in principle, on an adequate generalization of the de-Poissonization techniques of \cite{DynMan83} to the case of non-degenerate kernels whose expression possibly depends on the sample size, that should then be combined with the estimates from \cite{LRP2}. {In the general case of a sequence of non-degenerate $U$-statistics whose kernel varies with the sample size, implementing such an approach would involve a number of highly non-trivial technical difficulties: we therefore prefer to keep this direction of research as a separate subject of further investigation.} We stress that the intrinsic approach developed in the present paper will also yield some remarkable results of independent interest, most notably the product formulae stated in the next section.

\subsection{Plan} Section 2 contains some preliminary results, as well as a discussion of product formulae for degenerate $U$-statistics, and several useful estimates for contraction operators. Section 3 deals with one-dimensional approximation results for degenerate $U$-statistics, whereas the multidimensional case is dealt with in Section 4. In Section 5, we establish a number of new bounds for general $U$-statistics, whereas an application to random graphs is detailed in Section 6. Some technical proofs are collected in Section 7.

\section{Preliminary Notions and Auxiliary Results}\label{intro}

{We will now present several useful results concerning the Hoeffding decompositions of square-integrable $U$-statistics, as well as contraction operators. Both constitute the theoretical backbone of our approach.}
\smallskip

Every random object appearing in the sequel is defined on a suitable common probability space $(\Omega, \F, \Prob)$.

\subsection{Symmetric kernels and $U$-statistics}\label{ss:kernels}

Let $X_1,\dotsc,X_n$ be i.i.d. random variables taking values in a measurable space $(E,\mathcal{E})$ {(that we fix for the rest of this section)} and {denote by $\mu$ their common distribution}. For a fixed $p\in[n]:=\{1,\dotsc,n\}$, let 
\[\psi:\bigl(E^p,\mathcal{E}^{\otimes p}\bigr)\rightarrow\bigl(\R,\B(\R)\bigr)\] 
be a symmetric and measurable kernel of order $p$. 
By ``symmetric'' we mean that, for all $x=(x_1,\dotsc,x_p)\in E^p$ and each $\sigma\in\Sp$, the symmetric group acting on $\{1,\dotsc,p\}$, one has that 
\[\psi(x_1,\dotsc,x_p)=\psi(x_{\sigma(1)},\dotsc,x_{\sigma(p)})\,.\]
In general, the kernel $\psi$ might also depend on {the parameter} $n$, but we will {often suppress such a dependence, in order to simplify the notation}. 

\medskip

{In what follows, we will} write $X:=(X_i)_{1\leq i\leq n}$, {and use the symbol 
\begin{equation}\label{e:dp} \D_p:=\D_p(n):=\{J\subseteq[n]\,:\,\abs{J}=p\}
\end{equation}
to denote the collection of all $p$-subsets of $[n]$}. {For $p, \psi, X$ as above, we define}
\begin{align}\label{e:sus}
 J_p(\psi)&:=J_{p,X}(\psi):= \sum_{J\in\D_p}\psi(X_j,j\in J)=\sum_{1\leq i_1<\dotsc<i_p\leq n}\psi(X_{i_1},\dotsc,X_{i_p})\,.
\end{align}
We say that the random variable $ J_p(\psi)$ is the {\bf $U$-statistic of order $p$}, {\bf based on $X$ and generated by the kernel $\psi$}. For $p=0$ and a constant $c\in\R$ we further let $J_0(c):=0$. 

\medskip

Now assume that $p\geq1$ and $\psi\in L^1(\mu^{\otimes p})$. The kernel $\psi$ is called \textbf{(completely) degenerate} or \textbf{canonical} with respect to $\mu$, if 
\begin{equation*}
 \int_E\psi(x_1,x_2,\dotsc,x_p)d\mu(x_1)=0\quad\text{for } \mu^{\otimes p-1}\text{-a.a. }(x_2,\dotsc,x_p)\in E^{p-1}\,,
\end{equation*}
or, equivalently, if
\begin{equation*}
 \E\bigl[\psi(X_1,\dotsc,X_p)\,\bigl|\,X_1,\dotsc,X_{p-1}\bigr]=0\quad\Prob\text{-a.s.}
\end{equation*}

\begin{remark}{\rm 

In non-parametric statistics  (see e.g. the classical references \cite{KB94, Serf80}), the quantity 
\begin{equation*}
 U_p(\psi)=U_{p,X}(\psi):=\binom{n}{p}^{-1}J_p(\psi)=\binom{n}{p}^{-1}\sum_{1\leq i_1<\dotsc<i_p\leq n}\psi(X_{i_1},\dotsc,X_{i_p})
\end{equation*}
is called a \textit{$U$-statistic}, since it is always an \textit{unbiased} estimator of the parameter
\begin{equation*}
 \theta=\theta(\mu):=\E\bigl[\psi(X_1,\dotsc,X_p)\bigr]\,;
\end{equation*}
note that many well-known estimators from statistics turn out to be $U$-statistics (see again \cite{KB94, Serf80}). {We however choose to refer to the \textit{unaveraged version} $J_p(\psi)$ as a ``$U$-statistic''. Moreover, 
in this situation the kernel is typically not degenerate, since $\theta$ would have to be equal to $0$ otherwise. However, the centered kernel $\psi-\theta$ might well be degenerate.

}}
\end{remark}

\subsection{Hoeffding decompositions: general definition}\label{ss:hoeffding}

It is well-known {(see e.g. \cite{Hoeffding, Serf80, Vit92, DynMan83, LRP3})} that every {random variable}
\begin{equation*}
 Y=g(X_1,\dotsc,X_n)\in L^1(\Prob),
\end{equation*}
{having the form of a deterministic function $g$ of} (not necessarily identically distributed) independent random variables $X_1,\dotsc,X_n$, has a $\Prob$-a.s. unique representation {of the type}
\begin{equation}\label{HDgen}
Y=\sum_{M\subseteq[n]} Y_M =\sum_{s=0}^n\Biggl(\sum_{\substack{M\subseteq[n]:\\ \abs{M}=s}}Y_M\Biggr)
\end{equation}
{where, }for each $M\subseteq [n]$, the summand $Y_M$ is measurable with respect to $\F_M:=\sigma(X_j,j\in M)$ and, furthermore, 
\[\E[Y_M\,|\,\F_J]=0\quad\text{holds, whenever }M\nsubseteq J\,.\]
\noindent The representation \eqref{HDgen} is the {celebrated} \textbf{Hoeffding decomposition} of $Y$, {playing a fundamental role in many theoretical and applied problems involving the analysis of $U$-statistics in the large-sample limit, see again \cite{KB94, LRP3, Serf80} and the references therein}. The following explicit formula for the summands $Y_M$, $M\subseteq[n]$, is also {easily deduced from the exclusion-inclusion principle}:
\begin{equation*}
 Y_M=\sum_{J\subseteq M}(-1)^{\abs{M}-\abs{J}}\E[Y\,|\,\F_J]\,,
\end{equation*}
yielding in particular that $Y_\emptyset=\E[Y]$ a.s.-$\Prob$.

\subsection{Hoeffding decompositions for symmetric $U$-statistics}\label{ss:hoeffdingS}

Now assume that the {random variable} $Y$ is given by a $U$-statistic $J_p(\psi)$ {based on a vector $X = (X_1,..., X_n)$ of i.i.d. random variables, and generated by} a symmetric kernel $\psi$, {that is:
$$
Y = g(X_1,...,X_n) = J_{p,X}(\psi),
$$
where we used the notation \eqref{e:sus}}. In this case, the Hoeffding decomposition of $Y= J_p(\psi)$ can be expressed as the sum of its expectation and of a linear combination of $U$-statistics
generated by symmetric and degenerate kernels $\psi_s$ of orders $s=1,\dotsc,p$, {that is},
\begin{align}\label{HDsym}
 J_p(\psi)&=\E\bigl[J_p(\psi)\bigr]+\sum_{s=1}^p \binom{n-s}{p-s} J_s(\psi_s)=\sum_{s=0}^p \binom{n-s}{p-s} J_s(\psi_s)\notag\\
 &=\E\bigl[J_p(\psi)\bigr]+\sum_{s=1}^p \binom{n-s}{p-s} \sum_{1\leq i_1<\dotsc<i_s\leq n}\psi_s(X_{i_1},\dotsc,X_{i_s})\,,
\end{align}
where 
\begin{align}\label{defpsis}
\psi_s(x_1,\dotsc,x_s)&=\sum_{k=0}^s(-1)^{s-k}\sum_{1\leq i_1<\dotsc<i_k\leq s} g_k(x_{i_1},\dotsc,x_{i_k})
 \end{align}
and the symmetric functions $g_k:E^k\rightarrow\R$ are defined by 
\begin{equation}\label{gk}
 g_k(y_1,\dotsc,y_k):=\E\bigl[\psi(y_1,\dotsc,y_k,X_1,\dotsc,X_{p-k})\bigr]\,,
\end{equation}
in such a way that, for $1\leq s\leq p$, $\psi_s$ is symmetric and degenerate of order $s$. In particular, one has $g_0\equiv\psi_0\equiv \E\bigl[\psi(X_1,\dotsc,X_p)\bigr]$ and $g_p=\psi$.
For $s=1,\dotsc,p$ one has the alternative formula
\begin{align}\label{psis2}
 \psi_s(x_1,\dotsc,x_s)&=g_s(x_1,\dotsc,x_s)-\E\bigl[\psi(X_1,\dotsc,X_p)\bigr]\notag\\
 &\hspace{3cm}-\sum_{k=1}^{s-1} \sum_{1\leq i_1<\dotsc<i_k\leq s} \psi_k(x_{i_1},\dotsc,x_{i_k})\,,
\end{align}
{and one can easily check that the random variables $Y_M$ appearing in \eqref{HDgen} verify the relations }
\begin{equation*}
 Y_M=\bigl(J_p(\psi)\bigr)_M=\binom{n-\abs{M}}{p-\abs{M}}\psi_{\abs{M}}(X_j,j\in M)\,,\quad M\subseteq[n]\text{ s.t. }\abs{M}\leq p\,,
\end{equation*}
and $Y_M=0$ if $\abs{M}>p$.

\begin{remark}{\rm

Plainly, for the averaged version of the $U$-statistics, the Hoeffding decomposition reads
\begin{equation*}
 U_p(\psi)=\theta(\mu)+\sum_{s=1}^p \binom{p}{s}U_s(\psi_s)\,.
\end{equation*}
}

\end{remark}

\subsection{Analysis of Variance}\label{ss:var}

In this work, we are interested in symmetric $U$-statistics $Y = J_p(\psi)$, based on an i.i.d. sample $X$, such that the kernel $\psi$ is square-integrable with respect to $\mu^{\otimes p}$. Under such an assumption, the summands in the Hoeffding decomposition \eqref{HDsym} are orthogonal in $L^2(\Prob)$, thanks to the degeneracy of the 
kernels $\psi_s$, $s=1,\dotsc,p$. In particular, we have that
\begin{align}\label{varJp}
 \Var\bigl(J_p(\psi)\bigr)&=\sum_{s=1}^p \binom{n-s}{p-s}^2 \Var\bigl(J_s(\psi_s)\bigr)\notag\\
 &=\sum_{s=1}^p \binom{n-s}{p-s}^2\binom{n}{s}\Var\bigl(\psi_s(X_1,\dotsc,X_s)\bigr)\,.
\end{align}
Choosing $n=p$ leads to the following useful lower bound on the variance:
\begin{align}\label{varlb}
 \Var\bigl(\psi(X_1,\dotsc,X_p)\bigr)&=\sum_{s=1}^p\binom{p}{s}\Var\bigl(\psi_s(X_1,\dotsc,X_s)\bigr)\notag\\
 &\geq\Var\bigl(\psi_p(X_1,\dotsc,X_p)\bigr)\,.
\end{align}

Another useful variance formula in terms of the functions $g_k$ is as follows (see e.g. \cite[p. 183]{Serf80}):
\begin{align}\label{varJp2}
 \Var\bigl(J_p(\psi)\bigr)&= \binom{n}{p}\sum_{k=1}^p\binom{p}{k}\binom{n-p}{p-k}\Var\bigl(g_k(X_1,\dotsc,X_k)\bigr)\,.
\end{align}
Recalling that $g_p=\psi$ yields the following lower bound on the variance of $J_p(\psi)$:  
\begin{align}\label{varlb2}
 \Var\bigl(J_p(\psi)\bigr)&\geq \binom{n}{p}\Var\bigl(\psi(X_1,\dotsc,X_p)\bigr)\,.
\end{align}

\begin{remark}[On notation]{\rm For the rest of the paper, for every integer $m$ and every real $r>0$, we will use the standard notation:
$$
L^r(\mu^{\otimes m}) := L^r(E^m, \mathcal{E}^{\otimes m}, \mu^{\otimes m})
$$
Given a measurable mapping $\phi : E^m \to \R$, we will often write
$$
\| \phi\|_{L^r(\mu^{\otimes m})} := \left[ \int_{E^m} | \phi |^r \, d\mu^{\otimes m}\right]^{1/r}
$$
(by a slight abuse of notation), even when the right-hand side of the previous equation is infinite.
}
\end{remark}

\subsection{Contractions}\label{ss:contractions}

{We will now introduce one of the main analytical objects of the paper, that is, ``contraction kernels'' defined in term of pairs of square-integrable mappings}. For integers $p,q\geq1$, $0\leq l\leq r\leq p\wedge q$ and two symmetric kernels $\psi\in L^2(\mu^{\otimes p})$ and $\phi\in L^2(\mu^{\otimes q})$, define the \textbf{contraction kernel} $\psi\star_r^l \phi$ on $E^{p+q-r-l}$ by {the relation}
\begin{align}
 &(\psi\star_r^l \phi)(y_1,\dotsc,y_{r-l},t_1,\dotsc,t_{p-r},s_1,\dotsc,s_{q-r})\notag\\
&:=\int_{E^l}\Bigl(\psi\bigl(x_1,\dotsc,x_l, y_1,\dotsc,y_{r-l},t_1,\dotsc,t_{p-r}\bigr)\notag\\
&\hspace{3cm}\cdot\phi\bigl(x_1,\dotsc,x_l, y_1,\dotsc,y_{r-l},s_1,\dotsc,s_{q-r}\bigr)\Bigr)d\mu^{\otimes l}(x_1,\dotsc,x_l)\label{defcontr1}\\
&=\E\Bigl[\psi\bigl(X_1,\dotsc,X_l, y_1,\dotsc,y_{r-l},t_1,\dotsc,t_{p-r}\bigr)\notag\\
&\hspace{3cm}\cdot\phi\bigl(X_1,\dotsc,X_l, y_1,\dotsc,y_{r-l},s_1,\dotsc,s_{q-r}\bigr)\Bigr]\label{defcontr2}\, ,
\end{align}
{for every $(y_1,\dotsc,y_{r-l},t_1,\dotsc,t_{p-r},s_1,\dotsc,s_{q-r})$ belonging to the set $A_0 \subset E^{p+q-r-l}$ such that the right-hand side of the previous equation is a well-defined real number, and set it equal to zero otherwise}. Given $\psi, \varphi, r, l$ as above, we say that the kernel $\psi\star_r^l \phi$ is {\bf well-defined} if $\mu^{\otimes p+q-r-l}(A^c_0) =0$ (where $A_0$ is the set introduced in the previous sentence). Note that, in general, it is neither clear from the outset that $\psi\star_r^l \phi$ is well-defined in the sense specified above, nor that it is again square-integrable. 

\medskip

If $l=0$, then \eqref{defcontr1} is to be understood in the following way: 
\begin{align*}
 &(\psi\star_r^0 \phi)(y_1,\dotsc,y_{r},t_1,\dotsc,t_{p-r},s_1,\dotsc,s_{q-r})\notag\\
 &=\psi(y_1,\dotsc,y_{r},t_1,\dotsc,t_{p-r})\phi( y_1,\dotsc,y_{r},s_1,\dotsc,s_{q-r})\,.
\end{align*}
In particular, if $l=r=0$, then $\psi\star_r^l \phi$ reduces to the \textit{tensor product}
$$\psi\otimes \phi:E^{p+q}\rightarrow\R$$ of $\psi$ and $\phi$, given by 
\begin{align*}
 (\psi\otimes \phi)(x_1,\dotsc,x_{p+q})&=\psi(x_1,\dotsc,x_p)\cdot\phi(x_{p+1},\dotsc,x_{p+q})\,.
\end{align*}
Note also that $\psi\star_p^0\psi=\psi^2$ is square-integrable if and only if $\psi\in L^4(\mu^{\otimes p})$. Hence, $\psi\star_r^l\phi$ might not be in $L^2(\mu^{\otimes p+q-r-l})$ even though $\psi\in L^2(\mu^{\otimes p})$ and $\phi\in L^2(\mu^{\otimes q})$. Moreover, if $l=r=p$, then $\psi\star_p^p\psi=\|\psi\|_{L^2(\mu^{\otimes p})}^2$ is constant.

The next result lists the properties of contraction kernels that 
are most useful for the present work. The (quite technical) proof is deferred to Section \ref{proofs}.

\begin{lemma}\label{contrlemma}
Let $p,q\geq1$ be integers and fix two symmetric kernels $\psi\in L^2(\mu^{\otimes p})$ and $\phi\in L^2(\mu^{\otimes q})$.
\begin{enumerate}[{\normalfont (i)}]
 \item For all $0\leq l\leq r\leq p\wedge q$ the function $\psi\star_r^l \phi$ given by  \eqref{defcontr1} is well-defined, {in the sense specified at the beginning of the present subsection}.
 \item For all $0\leq l\leq r\leq p\wedge q$ one has that $$\displaystyle \|\psi\star_r^l\phi\|_{L^2(\mu^{\otimes p+q-r-l})}^2\leq  \|\psi\star_p^{p-r+l}\psi\|_{L^2(\mu^{\otimes r-l})} \cdot\|\phi\star_q^{q-r+l}\phi\|_{L^2(\mu^{\otimes r-l})},$$ 
 where both sides of the inequality might assume the value $+\infty$.
 \item For all $0\leq l\leq r\leq p\wedge q$ one has that $$\displaystyle \|\psi\star_r^l\phi\|_{L^2(\mu^{\otimes p+q-r-l})}^2\leq  \|\psi\star_p^{p-r}\psi\|_{L^2(\mu^{\otimes r})} \cdot\|\phi\star_q^{q-r}\phi\|_{L^2(\mu^{\otimes r})},$$ 
 where both sides of the inequality might assume the value $+\infty$.
 \item If $\psi\in L^4(\mu^{\otimes p})$ and $\phi\in L^4(\mu^{\otimes q})$, then, for all $0\leq r\leq p\wedge q$, one has $\psi\star_r^l \phi\in L^2(\mu^{\otimes p+q-r-l})$ and 
 $$\displaystyle \|\psi\star_r^l\phi\|_{L^2(\mu^{\otimes p+q-r-l})}\leq \|\psi\|_{L^4(\mu^{\otimes p})}\|\phi\|_{L^4(\mu^{\otimes q})}.$$
 \item For all $0\leq r\leq p\wedge q$ the function $\psi\star_r^r \phi$ is in $ L^2(\mu^{\otimes p+q})$ and
 $$\displaystyle \|\psi\star_r^r\phi\|_{L^2(\mu^{\otimes p+q-2r})}\leq \|\psi\|_{L^2(\mu^{\otimes p})}\|\phi\|_{L^2(\mu^{\otimes q})}.$$
\item If, for all $0\leq l\leq p-1$, $\psi\star_{p}^{l}\psi\in L^2(\mu^{\otimes p-l})$ and, for all $0\leq k\leq q-1$,  $\phi\star_{q}^{k}\phi\in L^2(\mu^{\otimes q-k})$, then, 
for all $0\leq l\leq r\leq p\wedge q$, one has 
$\psi\star_r^l \phi\in L^2(\mu^{\otimes p+q-r-l})$ and 
\begin{align*}
 \|\psi\star_r^l\phi\|_{L^2(\mu^{\otimes p+q-r-l})}^2&=\langle \psi\star_{p-l}^{p-r}\psi, \phi\star_{q-l}^{q-r}\phi\rangle_{L^2(\mu^{\otimes r+l})}\\
&\leq\|\psi\star_{r}^{l}\psi\|_{L^2(\mu^{\otimes 2p-r-l})}\cdot  \|\phi\star_{r}^{l}\phi\|_{L^2(\mu^{\otimes 2q-r-l})}<\infty\,.
\end{align*}
\end{enumerate}
\end{lemma}

\begin{remark}\label{contrem}
\begin{enumerate}[(a)]
\item We will heavily rely on item (iv) for deriving our normal approximation bounds. Moreover, in certain applications item (vi) (which is already contained in Lemma 2.9 of \cite{PZ1}) can be very useful in order to study the asymptotic distributional behaviour of vectors of degenerate $U$-statistics.
\item The contraction kernels defined by \eqref{defcontr1} also play a fundamental role for the normal approximation of functionals of a general Poisson measure {having the form of multiple Wiener-It\^o integrals or, more generally, of $U$-statistics} (see e.g. \cite{PSTU, LRP1, LRP2, PZ1, BPsv}),
as well as of functionals of a Rademacher sequence (see e.g. \cite{NPR10b, KRT1}). In these settings, the measure $\mu$ appearing in \eqref{defcontr1} is the control measure of the Poisson measure and the counting measure 
on $\N$, respectively, and, hence, it is in general not finite. We also stress that items (i), (ii), (v) and (vi) of Lemma \ref{contrlemma} also hold true for $\sigma$-finite measures $\mu$. This will be clear from the proof below.
\item Statements (iii) and (iv) can be {suitably adapted to the framework of a finite measure}, by introducing appropriate additional multiplicative constants on the right hand sides of the respective inequalities. 
For instance, inequality (iv) becomes\\
$\|\psi\star_r^l\phi\|_{L^2(\mu^{\otimes p+q-r-l})}\leq \mu(E)^{l-r+(p+q)/2}\|\psi\|_{L^4(\mu^{\otimes p})}\|\phi\|_{L^4(\mu^{\otimes q})}$. On the other hand, if $\mu(E)=+\infty$, then, in general, there is 
no finite constant $C=C(p,q,r,l)$ such that $\|\psi\star_r^l\phi\|_{L^2(\mu^{\otimes p+q-r-l})}\leq C\|\psi\|_{L^4(\mu^{\otimes p})}\|\phi\|_{L^4(\mu^{\otimes q})}$. Indeed, take $(E,\mathcal{E},\mu)=(\R,\B(\R),\lambda)$, $p=q=r=2$, $l=1$ and\\ 
$\psi(x,y)=\phi(x,y)=(1+x^2)^{-1/4}(1+y^2)^{-1/4}$. Then, 
\begin{align*}
 (\psi\star_{2}^{1}\psi)(x)&=\int_\R\psi(x,y)^2dy=(1+x^2)^{-1/2}\int_\R (1+y^2)^{-1/2}dy=+\infty
\end{align*}
for all $x\in\R$ and, a fortiori, $\|\psi\star_2^1\psi\|_{L^2(\lambda)}=+\infty$ but 
\begin{align*}
 \|\psi\|_{L^4(\lambda^{\otimes 2})}^4=\int_{\R^2} \frac{1}{1+x^2}\frac{1}{1+y^2}d\lambda^{\otimes2}(x,y)=\bigl(\arctan(x)|_{-\infty}^{+\infty}\bigr)^2=\pi^2<\infty\,.
\end{align*}
\end{enumerate}
\end{remark}

\subsection{Product formulae and related estimates}\label{ss:productintro}

{It is easily seen that} the contraction kernels $\psi\star_r^l\phi$ are, in general, not symmetric. If $f:E^p\rightarrow\R$ is an arbitrary function, then we denote by $\tilde{f}$ its \textit{canonical symmetrization} defined via
\begin{equation*}
\tilde{f}(x_1,\dotsc,x_p):=\frac{1}{p!}\sum_{\sigma\in\mathbb{S}_p} f(x_{\sigma(1)},\dotsc,x_{\sigma(p)})\,,
\end{equation*}
where, as before, $\mathbb{S}_p$ indicates the group of permutations of the set $[p]$. It easily follows from Minkowski's inequality that, if $f\in L^2(\mu^{\otimes p})$, then so is $\tilde{f}$ and 
\begin{equation}\label{normftilde}
\norm{\tilde{f}}_{L^2(\mu^{\otimes p})}\leq\norm{f}_{L^2(\mu^{\otimes p})}\,.
\end{equation}

\medskip

The following new formula for the product of two degenerate, symmetric $U$-statistics, {which has an independent interest}, will be crucial for the proofs of the main results provided in this work. 
Such a statement is a more explicit expression of the product formula for degenerate, not necessarily symmetric $U$-statistics which was provided recently in \cite{DP16}; {it also represents a particularly attractive alternative to the combinatorial product formulae for $U$-statistics derived in \cite[Chapter 11]{Maj13}}. The proof is provided in Section \ref{proofs}.
\begin{proposition}[Product formula for degenerate, symmetric $U$-statistics]\label{pform}
Let $p,q\geq1$ be positive integers and assume that $\psi\in L^2(\mu^{\otimes p})$ and $\phi\in L^2(\mu^{\otimes q})$ are degenerate, symmetric kernels of orders $p$ and $q$ respectively. Then, whenever $n\geq p+q$ we have the Hoeffding decomposition:
\begin{equation}\label{prodform}
 J_p(\psi) J_q(\phi)=\sum_{t=0}^{2(p\wedge q)} J_{p+q-t}(\chi_{p+q-t})\,,
\end{equation}
where, for $t\in\{0,1,\dotsc,2(p\wedge q)\}$, the degenerate, symmetric kernel $$\chi_{p+q-t}:E^{p+q-t}\rightarrow\R,$$ of order $p+q-t$, is given by 
\begin{equation}\label{kerpf}
 \chi_{p+q-t}=\sum_{r=\ceil{\frac{t}{2}}}^{t\wedge p\wedge q}\binom{n-p-q+t}{t-r}\binom{p+q-t}{p-r,q-r,2r-t}\bigl(\widetilde{\psi\star_r^{t-r}\phi}\bigr)_{p+q-t}\,.
\end{equation}
{In the previous expression,} the kernels $\bigl(\widetilde{\psi\star_r^{t-r}\phi}\bigr)_{p+q-t}$ appearing in the Hoeffding decomposition of $J_{p+q-t}\bigl((\widetilde{\psi\star_r^{t-r}\phi})_{p+q-t}\bigl)$
have been defined in \eqref{defpsis} and we have written $\ceil{x}$ to indicate the smallest integer greater or equal to the real number $x$.
\end{proposition}

\begin{remark}\label{pformrem}
\begin{enumerate}[(a)]
 \item Proposition \ref{pform} is in the same spirit as the existing product formulae for multiple stochastic integrals on the Wiener space (see e.g. Theorem 2.7.10 in \cite{NouPecbook}), on the Poisson space (see \cite{Sur84, Lastsv}) and for functionals of a Rademacher sequence --- see \cite{NPR10, Krok15, PrTo}. 
In particular, the product formula for multiple integrals on the Poisson space in its orthogonal form given explicitly by equation (19) of \cite{PZ1} is completely analogous to {\eqref{prodform}}, {as one can see by the change of variables} $k=p+q-t$ in \eqref{prodform}, and by replacing the indicator $\1_{\{p+q-r-l=k\}}$ in formula (18) of \cite{PZ1} with suitable conditions on the respective summation indices.
\item The product formula for non-symmetric and non-homogeneous Rademacher sequences in \cite[formula (5.3)]{PrTo} or, equivalently, in \cite[formula (2.4)]{Krok15} can be easily related to Proposition \ref{pform} and their similarity is quite striking. Note that, on the one hand, our formula is more general, in the sense that we allow for an arbitrary distribution of the underlying i.i.d. random variables whereas the formula in \cite{PrTo} is restricted to {discrete} multiple integrals which are functionals of a 
Rademacher sequence; on the other hand, the success parameters of the Rademacher sequences considered in \cite{PrTo} are allowed to vary and further the multiple integrals might depend on the whole infinite sequence.      
 \end{enumerate}
\end{remark}

In order to derive our main bounds, we will also make use of the following {elementary} lemmas.

\begin{lemma}\label{tele1}
For two positive integers $p,q\geq1$ assume that $\psi\in L^2(\mu^{\otimes p})$ and $\phi\in L^2(\mu^{\otimes q})$ are degenerate, symmetric kernels of orders $p$ and $q$, respectively. Then, for $t=1,\dotsc,2(p\wedge q)-1$, and with the kernels $\chi_{p+q-t}$ defined in \eqref{kerpf} we have 
\begin{align}\label{tl1}
\norm{\chi_{p+q-t}}_{L^2(\mu^{\otimes p+q-t})}&\leq\sum_{r=\ceil{\frac{t}{2}}}^{t\wedge p\wedge q}\binom{n-p-q+t}{t-r}\binom{p+q-t}{p-r,q-r,2r-t}\notag\\
&\hspace{4cm}\bigl\|\psi\star_r^{t-r}\phi\bigr\|_{L^2(\mu^{\otimes p+q-t})}\,.
\end{align}
\end{lemma}

\begin{proof}
From \eqref{kerpf} and \eqref{normftilde} we obtain 
\begin{align*}
&\norm{\chi_{p+q-t}}_{L^2(\mu^{\otimes p+q-t})}\\
&\leq\sum_{r=\ceil{\frac{t}{2}}}^{t\wedge p\wedge q}\binom{n-p-q+t}{t-r}\binom{p+q-t}{p-r,q-r,2r-t}\bigl\|\bigl(\widetilde{\psi\star_r^{t-r}\phi}\bigr)_{p+q-t}\bigr\|_{L^2(\mu^{\otimes p+q-t})}\notag\\
&\leq\sum_{r=\ceil{\frac{t}{2}}}^{t\wedge p\wedge q}\binom{n-p-q+t}{t-r}\binom{p+q-t}{p-r,q-r,2r-t}
\bigl\|\psi\star_r^{t-r}\phi\bigr\|_{L^2(\mu^{\otimes p+q-t})}\,,\notag
\end{align*}
which is the desired claim.
\end{proof}

\begin{lemma}\label{tele2}
Let $n,p,q,t,r$ be positive integers such that $n\geq p+q$ and $1\leq r\leq t\leq p+q-1$. Then, the inequality 
\begin{align*}
\frac{\sqrt{\binom{n}{p+q-t}}}{\sqrt{\binom{n}{p}}\sqrt{\binom{n}{q}}}\binom{n+t-p-q}{t-r}\binom{p+q-t}{p-r,q-r,2r-t}
&\leq C(p,q,t,r)\cdot n^{t/2 -r}
\end{align*}
is in order, where $C(p,q,t,r)$ is a suitable constant which only depends on $p,q,r$ and $t$.
\end{lemma} 

\begin{proof}
 This immediately follows from the definition of multinomial coefficients.
\end{proof}

\section{Main Results in Dimension One}\label{1dim}

\subsection{Degenerate $U$-statistics}\label{degU}
{For the rest of this} section, we let $Z\sim {N}(0,1)$ denote a standard normal random variable, {and write $X = (X_1,...,X_n)$ to indicate a vector of i.i.d. random variables, with values in a space $(E, \mathcal{E})$ and common distribution $\mu$}. We also fix a degenerate, symmetric kernel $\psi=\psi(n)$ {-- possibly depending on the integer parameter $n$ --} of order $p\geq 1$ {(see Section \ref{ss:kernels} for definitions)}, and we assume that 
\[\E\bigl[\psi^4(X_1,\dotsc,X_p)\bigr]<\infty\,.\]
{Writing $J_p(\psi)$ to indicate the $U$-statistic defined in \eqref{e:sus}}, the degeneracy of the kernel immediately implies that 
\begin{equation*}
 \E\bigl[J_p(\psi)\bigr]=0\quad\text{and}\quad \sigma_n^2:=\Var\bigl(J_p(\psi)\bigr)=\binom{n}{p}\E\bigl[\psi^2(X_1,\dotsc,X_p)\bigr]\,.
\end{equation*}
We assume that $\sigma_n^2>0$ and denote by $\phi:=\phi_n$ the kernel defined via
\begin{equation*}
 \phi_n(x_1,\dotsc,x_p):=
 \sigma_n^{-1}\psi(x_1,\dotsc,x_p)\,,\quad (x_1,\dotsc,x_p)\in E^p,
\end{equation*}
and let 
\begin{equation}\label{e:unorm}
 W:=W_n:=\sigma_n^{-1} J_p(\psi)=J_p(\phi)=\sum_{J\in\D_p}\phi(X_j,j\in J)\,,
\end{equation}
where the set $\D_p$ is defined in \eqref{e:dp}. Of course, $\E[W]=0$, $\Var(W)=1$ and, by degeneracy,
\begin{equation*}
 W=\sum_{J\in\D_p}W_J\quad\text{with}\quad W_J:=\phi(X_j,j\in J)\,,\quad J\in\D_p\,,
\end{equation*}
is the Hoeffding decomposition of $W$, as defined in Section \ref{ss:hoeffding}. Since, {by assumption}, $W$ is a square-integrable $U$-statistic of order $p$, it is easy to see that $U:=W^2$ admits a Hoeffding decomposition {of the type \eqref{HDgen}, that we write (with obvious notation) as} 
\begin{equation*}
 U=\sum_{\substack{M\subseteq[n]: \abs{M}\leq 2p}} U_M\,;
\end{equation*}
{the explicit form of the Hoeffding decomposition of $U$ can be of course be deduced from Proposition \ref{pform}.}

\medskip

{ 
\begin{definition}{\rm

Given two real-valued random variables $X, Y$ we write 
$$
d_\W(X,Y) := \sup_{h \in {\rm Lip}(1)} \left | \E[h(X)] - \E[h(Y)]\right |,
$$
where ${\rm Lip}(1)$ is the class of all 1-Lipschitz mappings $h : \R\to \R$, to indicate the {\bf Wasserstein distance} between the distributions of $X$ and $Y$ (see \cite[Appendix C]{NouPecbook}, and the references therein, for some basic properties of this distance). 
}
\end{definition}
}

The following lemma is a simple consequence of the techniques {developed} in \cite{DP16}. An outline of its proof is given in Section \ref{proofs}.

\begin{lemma}\label{1dimlemma}
{Under the notation of the present section, one has the bound} 
\begin{align}\label{wb2}
 d_\W(W,Z)&\leq \Bigl(\sqrt{\frac{2}{\pi}}+\frac43\Bigr)\Bigl(\sum_{\substack{M\subseteq[n]: \abs{M}\leq 2p-1}}\Var(U_M)\Bigr)^{1/2} +\frac{2\sqrt{2}}{3}\frac{\sqrt{p\kappa_p}}{\sqrt{n}}\,,
\end{align}
where $\kappa_p$ is a finite constant which only depends on $p$.
\end{lemma}

{We now state one of the main results of the paper. It corresponds to an explicit bound on the normal approximation of degenerate $U$-statistics, expressed in terms of contraction operators}.

\begin{theorem}\label{1dimbound}
With $W$ as defined above and with the constants $\kappa_p$ from Lemma \ref{1dimlemma} and $C(p,p,t,r)$ defined in Lemma \ref{tele2}, we have  
\begin{align}
d_\W(W,Z)&\leq \Bigl(\sqrt{\frac{2}{\pi}}+\frac43\Bigr)\sum_{t=1}^{2p-1}\sum_{r=\ceil{\frac{t}{2}}}^{t\wedge p} C(p,p,t,r)
\frac{\|\psi\star_r^{t-r}\psi\|_{L^2(\mu^{\otimes 2p-t})}}{\|\psi\|_{L^2(\mu^{\otimes p})}^2} n^{t/2-r}\notag\\
&\;+\frac{2\sqrt{2}}{3}\frac{\sqrt{p\kappa_p}}{\sqrt{n}}\label{1dimb1}\\
&\leq \frac{2\sqrt{2}}{3}\frac{\sqrt{p\kappa_p}}{\sqrt{n}}+\Bigl(\sqrt{\frac{2}{\pi}}+\frac43\Bigr)\biggl(\sum_{s=1}^{p-1}C(p,p,2s,s) \frac{\|\psi\star_s^{s}\psi\|_{L^2(\mu^{\otimes 2p-2s})}}{\|\psi\|_{L^2(\mu^{\otimes p})}^2}\notag\\
&\;+\frac{\|\psi\|_{L^4(\mu^{\otimes p})}^2}{\|\psi\|_{L^2(\mu^{\otimes p})}^2}\biggl(\sum_{s=1}^{p-1} \sum_{r=s+1}^{(2s)\wedge p} C(p,p,2s,r)n^{s-r}\notag\\
&\hspace{3cm}+\sum_{s=1}^p \sum_{r=s}^{(2s-1)\wedge p}C(p,p,2s-1,r)n^{s-r-1/2}\biggr)\biggr).\label{1dimb2}
\end{align}
\end{theorem}

\begin{remark}\label{r:orders}{\rm {Fix $p$, and assume as before that the kernel $\psi=\psi(n)$ depends on the parameter $n$.} Then, as $n\to\infty$, the bound \eqref{1dimb2} is of the order
\begin{align*}
 &O\bigl(n^{-1/2}\|\psi\|_{L^4(\mu^{\otimes p})}^2\|\psi\|_{L^2(\mu^{\otimes p})}^{-2}\bigr)\notag\\
&\;+O\Bigl(\|\psi\|_{L^2(\mu^{\otimes p})}^{-2}\max_{1\leq s\leq p-1}\|\psi\star_s^{s}\psi\|_{L^2(\mu^{\otimes 2p-2s})}\Bigr)+O\bigl(n^{-1/2}\bigr),
\end{align*}
whereas the bound \eqref{1dimb1} {behaves asymptotically as}
\begin{align*}
 &O\Biggl(\|\psi\|_{L^2(\mu^{\otimes p})}^{-2}\max_{1\leq r\leq p-1}\|\psi\star_r^{r}\psi\|_{L^2(\mu^{\otimes 2p-2r})}\Biggr)+O\bigl(n^{-1/2}\bigr)\\
&\; +O\Biggl(\|\psi\|_{L^2(\mu^{\otimes p})}^{-2}\max_{\substack{1\leq r\leq p\,,\\0\leq l\leq r-1}}n^{(l-r)/2}\|\psi\star_r^{l}\psi\|_{L^2(\mu^{\otimes 2p-r-l})}\Biggr)\,.
\end{align*}
}
\end{remark}

{The asymptotic relations pointed out in Remark \ref{r:orders} immediately yield the following one-dimensional CLT}.

\begin{corollary}\label{1dimcor}
Let $p$ be a fixed positive integer and, for each $n\geq p$, let $\psi(n)\in L^4(\mu^{\otimes p})$ be a symmetric and degenerate kernel with respect to the probability measure $\mu$ such that $\|\psi(n)\|_{L^2(\mu^{\otimes p})}>0$. Let $X_1,X_2,\dotsc$ be i.i.d. random variables on $(\Omega,\F,\Prob)$ with {common} distribution $\mu$ and, for $n\geq p$, let $W_n$ be the normalized random variable {obtained from}
\[J_p(\psi(n)):=\sum_{1\leq i_1<\dotsc<i_p\leq n}\psi(n)(X_{i_1},\dotsc,X_{i_p})\,,\]
{ according to the definition \eqref{e:unorm}}. Assume that the following conditions {\normalfont (i)} and {\normalfont (ii)} are satisfied:
\begin{enumerate}[{\normalfont (i)}]
\item For all $1\leq s\leq p-1$, one has that $\displaystyle \lim_{n\to\infty}\frac{\|\psi(n)\star_s^{s}\psi(n)\|_{L^2(\mu^{\otimes 2p-2s})}}{\|\psi(n)\|_{L^2(\mu^{\otimes p})}^2}=0$.
\item We have that $\displaystyle \lim_{n\to\infty}\frac{1}{\sqrt{n}}\frac{\|\psi(n)\|^2_{L^4(\mu^{\otimes p})}}{\|\psi(n)\|^2_{L^2(\mu^{\otimes p})}}=0$.
\end{enumerate}
Then, as $n\to\infty$, $W_n$ converges in distribution to $Z\sim N(0,1)$.
\end{corollary}

\begin{remark}\label{1dimrem}
\begin{enumerate}[(a)]
\item The statement of Corollary \ref{1dimcor} is in fact an extension of a CLT by Hall \cite{Hall84} to general $p$. 
Indeed, {in this reference} it is proved that, with the above notation for $p=2$, the CLT for $W_n$, $n\in\N$, holds, whenever 
$$ \lim_{n\to\infty}\frac{1}{n}\frac{\E\bigl[\psi(n)^4(X_1,X_2)\bigr]}{\bigl(\E\bigl[\psi(n)^2(X_1,X_2)\bigr]\bigr)^2}=0\quad\text{and}\quad 
\lim_{n\to\infty}\frac{\|\psi(n)\star_1^{1}\psi(n)\|_{L^2(\mu^{\otimes 2})}}{\|\psi(n)\|_{L^2(\mu^{\otimes 2})}^{2}} =0\,.$$
Note that our bound \eqref{1dimb2} even gives a precise estimate of the error of normal approximation in this situation.
\item If, in the situation of Corollary \ref{1dimcor}, the kernel $\psi=\psi(n)$ is fixed, i.e. it does not depend on $n$, then { our results imply a CLT} if and only if $p=1$, i.e., if we are dealing with a sum of i.i.d. random variables. 
This is in accordance with classical results about the non-Gaussian fluctuations of degenerate $U$-statistics with a fixed kernel of order $p\geq2$ \cite{Greg77, Serf80, DynMan83, LRP2}.
\end{enumerate}
\end{remark}

\begin{proof}[Proof of Theorem \ref{1dimbound}]
{ We will apply} Lemma \ref{1dimlemma}. { Our goal is therefore} to effectively bound from above the quantity 
\begin{equation*}
 \sum_{\substack{M\subseteq[n]: \abs{M}\leq 2p-1}}\Var(U_M),
\end{equation*}
in terms of the kernel function $\psi$. From Proposition \ref{pform}, we deduce that 
\begin{align*}
 \sum_{\abs{M}\leq 2p-1}\Var(U_M)&=\sum_{s=0}^{2p-1}\Var\Biggl(\sum_{\abs{M}=s} U_M\Biggr)\\
&= \sum_{t=1}^{2p} \Var\bigl(J_{2p-t}(\chi_{2p-t})\bigr)=\sum_{t=1}^{2p-1} \Var\bigl(J_{2p-t}(\chi_{2p-t})\bigr)\,,
\end{align*}
where 
\begin{equation*}
 \chi_{2p-t}=\sum_{r=\ceil{\frac{t}{2}}}^{t\wedge p}\binom{n-2p+t}{t-r}\binom{2p-t}{p-r,p-r,2r-t}\bigl(\widetilde{\phi\star_r^{t-r}\phi}\bigr)_{2p-t}\,.
\end{equation*}

Using \eqref{varJp}, \eqref{varlb} as well as Lemmas \ref{tele1} and \ref{tele2}, for a fixed $t\in\{1,\dotsc,2p-1\}$, we obtain that 
\begin{align}\label{1de1}
 &\Bigl(\Var\bigl(J_{2p-t}(\chi_{2p-t})\bigr)\Bigr)^{1/2}=\Biggl(\binom{n}{2p-t}\Var\bigl(\chi_{2p-t}(X_1,\dotsc,X_{2p-t})\bigr)\Biggr)^{1/2}\notag\\
 &=\sqrt{\binom{n}{2p-t}}\|\chi_{2p-t}\|_{L^2(\mu^{\otimes 2p-t})}\notag\\
 &\leq\sqrt{\binom{n}{2p-t}}\sum_{r=\ceil{\frac{t}{2}}}^{t\wedge p}\binom{n-2p+t}{t-r}\binom{2p-t}{p-r,p-r,2r-t}
\|\phi\star_r^{t-r}\phi\|_{L^2(\mu^{\otimes 2p-t})}\notag\\
 &=\frac{\sqrt{\binom{n}{2p-t}}}{\binom{n}{p}\|\psi\|_{L^2(\mu^{\otimes p})}^2}\sum_{r=\ceil{\frac{t}{2}}}^{t\wedge p}\binom{n-2p+t}{t-r}\binom{2p-t}{p-r,p-r,2r-t}\|\psi\star_r^{t-r}\psi\|_{L^2(\mu^{\otimes 2p-t})}\notag\\
&\leq \frac{1}{\|\psi\|_{L^2(\mu^{\otimes p})}^2}\sum_{r=\ceil{\frac{t}{2}}}^{t\wedge p} C(p,p,t,r)
\|\psi\star_r^{t-r}\psi\|_{L^2(\mu^{\otimes 2p-t})} n^{t/2-r}\,.
\end{align}
Using Lemma \ref{contrlemma}-(iv)  for $\frac{t}{2}<r\leq t\wedge p$ and distinguishing the cases of even and odd values of $t$, we thus infer the chain of inequalities 
\begin{align}
\Bigl(\sum_{\abs{M}\leq 2p-1}\Var(U_M)\Bigr)^{1/2}&\leq \sum_{t=1}^{2p-1} \Bigl(\Var\bigl(J_{2p-t}(\chi_{2p-t})\bigr)\Bigr)^{1/2}\notag\\
&\leq \sum_{t=1}^{2p-1}\sum_{r=\ceil{\frac{t}{2}}}^{t\wedge p} C(p,p,t,r)
\frac{\|\psi\star_r^{t-r}\psi\|_{L^2(\mu^{\otimes 2p-t})}}{\|\psi\|_{L^2(\mu^{\otimes p})}^2} n^{t/2-r}\label{1de2}\\
&\leq \sum_{s=1}^{p-1}C(p,p,2s,s) \frac{\|\psi\star_s^{s}\psi\|_{L^2(\mu^{\otimes 2p-2s})}}{\|\psi\|_{L^2(\mu^{\otimes p})}^2}\notag\\
&\; +\sum_{s=1}^{p-1} \sum_{r=s+1}^{(2s)\wedge p} C(p,p,2s,r)\frac{\|\psi\|_{L^4(\mu^{\otimes p})}^2}{\|\psi\|_{L^2(\mu^{\otimes p})}^2}\,n^{s-r}\notag\\
&\;+\sum_{s=1}^p \sum_{r=s}^{(2s-1)\wedge p}C(p,p,2s-1,r)\frac{\|\psi\|_{L^4(\mu^{\otimes p})}^2}{\|\psi\|_{L^2(\mu^{\otimes p})}^2}\,n^{s-r-1/2}\,.\label{1de3}
\end{align}
The bounds \eqref{1dimb1} and \eqref{1dimb2} now follow from Lemma \ref{1dimlemma} and from the bounds \eqref{1de2} 
and \eqref{1de3}, respectively.
\end{proof}


\subsection{$U$-statistics with a dominant component}\label{genU}
In this subsection we drop the restriction that the kernel $\psi$ be degenerate, { and we obtain quantitative CLTs under the assumptions that one of the terms in the Hoeffding decomposition is dominant in the large sample limit $n\to \infty$. The reason why we treat this case separately from the general results of Section \ref{general} is that, by virtue of the one-dimensional results of the previous section, we are able to obtain explicit bounds in the Wasserstein distance. The theory developed in Section \ref{general} will hinge on multidimensional results involving smooth distances, and will therefore yield bounds for more regular test functions.} 

\medskip

We now assume that $\psi=\psi(n)$ is a symmetric kernel of a fixed order $1\leq p\leq n$ such that 
\[0<\E\bigl[\psi^4(X_1,\dotsc,X_p)\bigr]<\infty\,.\]
Denote by 
\begin{equation*}
 J_p(\psi)=\E\bigl[J_p(\psi)\bigr]+\sum_{s=1}^p \binom{n-s}{p-s} J_s(\psi_s)
\end{equation*}
the Hoeffding decopmposition \eqref{HDsym} of $J_p(\psi)$ with symmetric and degenerate kernels $\psi_s$ of order $s$ which automatically satisfy 
\[\E\bigl[\psi_s^4(X_1,\dotsc,X_s)\bigr]<\infty\,,\]
$s=1,\dotsc,p$. This can be easily seen from their explicit construction. Let us further assume w.l.o.g. that $\E[J_p(\psi)]=0$ and that $\|\psi\|^2_{L^2(\mu^{\otimes p})}=\Var(\psi(X_1,\dotsc,X_p))=1$. We then define 
\begin{align}\label{Hrank}
 m&:=\min\{1\leq s\leq p\,:\,\Var\bigl(\psi_s(X_1,\dotsc,X_s)\bigr)\not=0\}\notag\\
 &\,=\min\{1\leq k\leq p\,:\,\Var\bigl(g_k(X_1,\dotsc,X_k)\bigr)\not=0\}
\end{align}
to be the so-called \textbf{order of degeneracy} or \textbf{Hoeffding rank} of $J_p(\psi)$. The second equality in \eqref{Hrank} easily follows from \eqref{defpsis} and \eqref{psis2}.
Let 
\begin{equation*}
 \sigma_m^2:=\Var\Biggl(\binom{n-m}{p-m}J_m(\psi_m)\Biggr)=\binom{n-m}{p-m}^2\binom{n}{m}\|\psi_m\|_{L^2(\mu^{\otimes m})}^{2}
\end{equation*}
as well as 
\begin{equation*}
 W:=\sigma_m^{-1}J_p(\psi)=\frac{J_m(\psi_m)}{\binom{n}{m}^{1/2}\|\psi_m\|_{L^2(\mu^{\otimes m})}}+\sigma_m^{-1}\sum_{s=m+1}^p\binom{n-s}{p-s} J_s(\psi_s)=:Y+R\,
\end{equation*}
{ (note that $W,Y,R$ all implicitly depend on $n$)}. We provide the following bound on the Wasserstein distance between the law of $W$ and the standard normal distribution,
which is useful whenever the random variable $Y$ { (that is, the first non-trivial Hoeffding component of $W$)} is dominant and $R$ is negligible.

\begin{theorem}\label{1dimth2}
Under the above assumption, one has the estimates 
\begin{align*}
d_\W(W,Z)&\leq d_\W(Y,Z)+\sum_{s=m+1}^p\frac{\sqrt{m!}(p-m)!\|\psi_s\|_{L^2(\mu^{\otimes s})}}{\sqrt{s!}(p-s)!\|\psi_m\|_{L^2(\mu^{\otimes m})}}n^\frac{m-s}{2}\\
&\leq d_\W(Y,Z)+\sum_{s=m+1}^p\frac{\sqrt{m!}(p-m)!}{\sqrt{p!}\sqrt{(p-s)!}\|\psi_m\|_{L^2(\mu^{\otimes m})}}n^\frac{m-s}{2}\,,
\end{align*}
and suitable bounds on $d_\W(Y,Z)$ are provided by Theorem \ref{1dimbound}.
\end{theorem}

\begin{proof}
Using the simple inequality 
\begin{equation*}
 d_\W(W,Z)\leq d_\W(Y,Z)+\sqrt{\Var(R)}
\end{equation*}
as well as 
\begin{align*}
 &\sigma_m^{-2}\Var\Biggl(\binom{n-s}{p-s}J_s(\psi_s)\Biggr)=\frac{\binom{n-s}{p-s}^2\binom{n}{s}\|\psi_s\|_{L^2(\mu^{\otimes s})}^{2}}{\binom{n-m}{p-m}^2\binom{n}{m}\|\psi_m\|_{L^2(\mu^{\otimes m})}^{2}}\\
 &=\frac{m!((p-m)!)^2\|\psi_s\|_{L^2(\mu^{\otimes s})}^{2}}{s!((p-s)!)^2\|\psi_m\|_{L^2(\mu^{\otimes m})}^{2}}\frac{1}{(n-m)\cdot(n-m-1)\cdot\ldots\cdot(n-s+1)}\\
 &\leq \frac{m!((p-m)!)^2\|\psi_s\|_{L^2(\mu^{\otimes s})}^{2}}{s!((p-s)!)^2\|\psi_m\|_{L^2(\mu^{\otimes m})}^{2}}(n-p+1)^{m-s}
\end{align*}
we obtain that 
\begin{align*}
 d_\W(W,Z)&\leq d_\W(Y,Z)+\Biggl(\sum_{s=m+1}^p\frac{m!((p-m)!)^2\|\psi_s\|_{L^2(\mu^{\otimes s})}^{2}}{s!((p-s)!)^2\|\psi_m\|_{L^2(\mu^{\otimes m})}^{2}}(n-p+1)^{m-s}\Biggr)^{1/2}\\
 &\leq d_\W(Y,Z)+\sum_{s=m+1}^p\frac{\sqrt{m!}(p-m)!\|\psi_s\|_{L^2(\mu^{\otimes s})}}{\sqrt{s!}(p-s)!\|\psi_m\|_{L^2(\mu^{\otimes m})}}(n-p+1)^\frac{m-s}{2}\\
 &\leq d_\W(Y,Z)+\sum_{s=m+1}^p\frac{\sqrt{m!}(p-m)!\|\psi_s\|_{L^2(\mu^{\otimes s})}}{\sqrt{s!}(p-s)!\|\psi_m\|_{L^2(\mu^{\otimes m})}}n^\frac{m-s}{2}\,.
\end{align*}
This is the first bound stated in the Theorem. The second one follows immediately from this one and from \eqref{varlb} since 
\[\|\psi_s\|^2_{L^2(\mu^{\otimes s})}\leq\frac{s!(p-s)!}{p!}\|\psi\|^2_{L^2(\mu^{\otimes p})}=\frac{s!(p-s)!}{p!}\,.\]
\end{proof}

Combined with Theorem \ref{1dimbound}, Theorem \ref{1dimth2} yields the following CLT.

\begin{corollary}\label{1dclt2}
Let $p$ be a fixed positive integer and, for each $n\geq p$, let $\psi(n)\in L^4(\mu^{\otimes p})$ be a symmetric kernel such that $\int_{E^p}\psi\, d\mu^{\otimes p}=0$ and $\|\psi(n)\|_{L^2(\mu^{\otimes p})}>0$. 
Let $X_1,X_2,\dotsc$ be i.i.d. random variables on $(\Omega,\F,\Prob)$ with distribution $\mu$ and, for $n\geq p$, denote by $m=m_n$ the Hoeffding rank of the $U$-statistic
\[J_p(\psi(n)):=\sum_{1\leq i_1<\dotsc<i_p\leq n}\psi(n)(X_{i_1},\dotsc,X_{i_p})=\sum_{s=0}^p \binom{n-s}{p-s} J_s(\psi(n)_s)\] 
and let 
\begin{equation*}
 \sigma_{m_n}^2:=\Var\Biggl(\binom{n-m_n}{p-m_n}J_{m_n}(\psi(n)_{m_n})\Biggr)=\binom{n-m_n}{p-m_n}^2\binom{n}{m_n}
\|\psi(n)_{m_n}\|_{L^2(\mu^{\otimes m_n})}^{2}\,.
\end{equation*}
Then, with $W_n:=\sigma_{m_n}^{-1}J_p(\psi(n))$, $n\geq p$, assume that the following conditions {\normalfont (i)}, {\normalfont (ii)} and {\normalfont (iii)} are satisfied:
\begin{enumerate}[{\normalfont (i)}]
\item We have $\displaystyle \lim_{n\to\infty}\max_{1\leq s\leq m_n-1}\frac{\|\psi(n)_{m_n}\star_s^{s}\psi(n)_{m_n}\|_{L^2(\mu^{\otimes 2m_n-2s})}}{\|\psi(n)_{m_n}\|_{L^2(\mu^{\otimes m_n})}^2}=0$.
\item We have $\displaystyle \lim_{n\to\infty}\frac{1}{\sqrt{n}}\frac{\|\psi(n)_{m_n}\|^2_{L^4(\mu^{\otimes m_n})}}{\|\psi(n)_{m_n}\|_{L^2(\mu^{\otimes m_n})}}=0$.
\item We have $\displaystyle \lim_{n\to\infty}\frac{\|\psi(n)\|_{L^2(\mu^{\otimes p})}}{\sqrt{n}\|\psi(n)_{m_n}\|_{L^2(\mu^{\otimes m_n})}}=0$.
\end{enumerate}
Then, as $n\to\infty$, $W_n$ converges in distribution to $Z\sim N(0,1)$.\\
If the Hoeffding rank $m=m_n$ does in fact not depend on $n$, then {\normalfont(iii)} can be replaced with the weaker condition 
\begin{enumerate}
 \item[{\normalfont(iii)$^\prime$}] For all $s=m+1,\dotsc,p$: $\displaystyle\lim_{n\to\infty} n^{\frac{m-s}{2}}\frac{\|\psi(n)_s\|_{L^2(\mu^{\otimes s})}}{\|\psi(n)_{m}\|_{L^2(\mu^{\otimes m})}}=0$.
\end{enumerate}

\end{corollary}

{ Again,} if we are dealing with a fixed kernel $\psi$ { not depending on $n$}, then also $m$ does not depend on $n$ and we obtain asymptotic normality of $W_n$ if and only if $m=1$. As observed before, such a phenomenon is consistent with classical results about the asymptotic distribution of $U$-statistics, see e.g. \cite{Hoeffding, Greg77, Serf80, DynMan83}.

\section{Multivariate Results}\label{mult}
{ Our goal in this section is to deduce explicit multidimensional bounds for vectors of degenerate $U$-statistics. As in the previous section, we denote by $X = (X_1,...,X_n)$ ($n\geq 1$) a vector of i.i.d. random variables, with values in $(E, \mathcal{E})$ and with distribution $\mu$}. 

\medskip

\subsection{ Setup}\label{ss:multisetup}

We start by fixing a positive integer $d$ and, for $1\leq i\leq d$, we let $\psi^{(i)}=\psi^{(n,i)}$ be a degenerate and symmetric kernel of order $1\leq p_i\leq n$ { (as before, the tacit dependence of the kernels on the sample size $n$ will be omitted whenever there is no risk of confusion)}. 
We will again assume that $\psi^{(i)}\in L^4(\mu^{\otimes p_i})$ and, for $1\leq i\leq d$, define 
\[\phi^{(i)}:=\phi^{(n,i)}:=\frac{\psi^{(i)}}{\sqrt{\binom{n}{p_i}}}\quad\text{as well as}\]
\begin{equation*}
 \sigma_n(i)^2:=\Var\bigl(J_{p_i}(\phi^{(i)})\bigr)=\binom{n}{p_i}\E\bigl[\bigl(\phi^{(i)}\bigr)^2(X_1,\dotsc,X_{p_i})\bigr]=\|\psi^{(n,i)}\|^2_{L^2(\mu^{\otimes p})}\,.
\end{equation*}
For $i=1,\dotsc,d$ write
\begin{equation*}
 W(i):=J_{p_i}(\phi^{(i)})
\end{equation*}
and let 
\begin{equation*}
 W:=\bigl(W(1),\dotsc,W(d)\bigr)^T\,.
\end{equation*}
 Without loss of generality, we can assume that $p_i\leq p_k$ whenever $1\leq i<k\leq d$.
Thus, there is an $s\in\{1,\dotsc,d\}$ as well as positive integers $1\leq d_1<d_2<\dotsc <d_s=d$ and $1\leq q_1<q_2<\ldots<q_s$ such that 
\begin{equation*}
 p_i=q_l\quad\text{for all}\quad i\in\{d_{l-1}+1,\dotsc,d_l\}\quad\text{and all}\quad l=1,\dotsc,s\,,
\end{equation*}
where $d_0:=0$. We also let 
\begin{equation*}
 v_{i,k}:=\Cov\bigl(W(i),W(k)\bigr)=\E\bigl[W(i)W(k)\bigr]\,,\quad 1\leq i\leq k\leq d\,,
\end{equation*}
and 
\[\V =\V(W) :=\Cov(W)=(v_{i,k})_{1\leq i,k\leq d}\,.\]
Note that $v_{i,i}=\sigma_n(i)^2$ for $i=1,\dotsc,r$ and $\abs{v_{i,k}}\leq\sigma_n(i)\sigma_n(k)$ for $1\leq i,k\leq d$, by the Cauchy-Schwarz inequality. 
Note also that, by degeneracy of the kernels, $v_{i,k}=0$ unless $p_i=p_k$. Hence, $\V$ is a block diagonal matrix. 
Throughout this section we denote by 
\begin{equation*}
Z=\bigl(Z(1),\dotsc,Z(d)\bigr)^T\sim N_d(0,\V)
\end{equation*}
a centered Gaussian vector with covariance matrix $\V$. Furthermore, for $1\leq i,k\leq d$, we denote by 
\begin{equation*}
 W(i)W(k)=\sum_{\substack{M\subseteq[n]: \abs{M}\leq p_i+p_k}} U_M(i,k)
\end{equation*}
the Hoeffding decomposition of $W(i)W(k)$; { similarly to the situation of the previous section, the explicit form of the random variables $U_M(i,k)$ can be deduced from Proposition \ref{pform}.}

\subsection{ Generalities on matrix norms and related estimates}\label{ss:matnorms}

For a vector $x=(x_1,\dotsc,x_d)^T\in\R^d$ we denote by $\Enorm{x}$ its \textit{Euclidean norm} and for a matrix $A\in\R^{d\times d}$ we let $\Opnorm{A}$ be
the \textbf{operator norm} induced by the Euclidean norm, i.e., 
\[\Opnorm{A}:= \sup\{ \Enorm{Ax}\,:\ \Enorm{x} = 1\}\,.\]
More generally, for any $k$-multilinear form $\psi:(\R^d)^k\rightarrow\R$, $k\in\N$, we define its \textbf{ (generalized) operator norm} as
\[\Opnorm{\psi}:=\sup\left\{\abs{\psi(u_1,\ldots,u_k)}\,:\, u_j\in\R^d,\, \Enorm{u_j}=1,\, j=1,\ldots,k\,\right\}.\]
Recall that for a function $h:\R^d\rightarrow\R$, its minimum Lipschitz constant $M_1(h)$ is given by
\[M_1(h):=\sup_{x\not=y}\frac{\abs{h(x)-h(y)}}{\Enorm{x-y}}\in[0,\infty)\cup\{\infty\}.\]
If, for instance, $h$ is differentiable, then it is easy to see that $$M_1(h)=\sup_{x\in\R^d}\Opnorm{Dh(x)}$$ 
If, more generally, $k\geq1$ and if $h:\R^d\rightarrow\R$ is a $(k-1)$-times differentiable function, then we let
\[M_k(h):=\sup_{x\not=y}\frac{\Opnorm{D^{k-1}h(x)-D^{k-1}h(y)}}{\Enorm{x-y}}\,,\]
thus viewing the $(k-1)$-th derivative $D^{k-1}h$ of $h$ at any point $x\in\R^d$ as a $(k-1)$-multilinear form.
Then, if $h$ is actually $k$-times differentiable, we have $M_k(h)=\sup_{x\in\R^d}\Opnorm{D^kh(x)}$. 
Thus, for $k=0$, we also define  $M_0(h):=\fnorm{h}$.

\medskip

Recall that the \textbf{Hilbert-Schmidt inner product} of two matrices $A,B\in\R^{d\times d}$ is defined by 
\begin{eqnarray*}
\langle A,B\rangle_{\HS}:=\Tr\bigl(AB^T\bigr)=\Tr\bigl(BA^T\bigr)=\Tr\bigl(B^TA\bigr)=\sum_{i,j=1}^d a_{ij}b_{ij}\,.
\end{eqnarray*}
Thus, $\langle\cdot,\cdot\rangle_{\HS}$ is just the standard inner product on $\R^{d\times d}\cong \R^{d^2}$.
The corresponding \textbf{Hilbert-Schmidt norm} will be denoted by $\HSnorm{\cdot}$.  
With this notion at hand, and following \cite{ChaMe08} and \cite{Meck09}, for $k=2$ we finally define 
\[\tilde{M}_2(h):=\sup_{x\in\R^d}\HSnorm{\Hess h(x)}\,,\]
with $\Hess h$ being the \textit{Hessian matrix} of $h$. Then, we have the inequality 
\begin{equation}\label{M2bound}
 \tilde{M}_2(h)\leq \sqrt{d} M_2(h)\,.
\end{equation}

\subsection{ Main results}\label{ss:mainmulti}

The next lemma is the multivariate counterpart to Lemma \ref{1dimlemma} and, as the latter, relies on the methods provided 
in the recent paper \cite{DP16}. Its proof is sketched in Section \ref{proofs}.

\begin{lemma}\label{mdimlemma}
 Under the { assumptions of Section \ref{ss:multisetup}}, the following holds. There are constants $\kappa_{p_i}\in(0,\infty)$, only depending on $p_i$, $1\leq i\leq d$, such that:
 \begin{enumerate}[{\normalfont(i)}]
  \item For any $h\in C^3(\R^d)$ such that $\E\bigl[\abs{h(W)}\bigr]<\infty$ and $\E\bigl[\abs{h(Z)}\bigr]<\infty$, 
\begin{align*}
&\babs{\E[h(W)]-\E[h(Z)]}\leq \frac{1}{4p_1}\tilde{M}_2(h)\sum_{i,k=1}^d(p_i+p_k)\biggl(\sum_{\substack{M\subseteq[n]:\\\abs{M}\leq p_i+p_k-1}}\Var\bigl(U_M(i,k)\bigr)\biggr)^{1/2}\\
&\;+\frac{2M_3(h)\sqrt{d}}{9p_1}\sum_{i=1}^d p_i\sigma_n(i)\biggl(\sum_{\substack{M\subseteq[n]:\\\abs{M}\leq 2p_i-1}}\Var\bigl(U_M(i,i)\bigr)\biggr)^{1/2}\\
&\;+\frac{\sqrt{2d}M_3(h)}{9p_1\sqrt{n}}\sum_{i=1}^d p_i^{3/2}\sigma_n(i)^3\sqrt{\kappa_{p_i}}\,.
\end{align*}
\item If, moreover, $\V$ is positive definite, then for each $h\in C^2(\R^d)$ such that\\ $\E\bigl[\abs{h(W)}\bigr]<\infty$ and $\E\bigl[\abs{h(Z)}\bigr]<\infty$, 
\begin{align*}
&\babs{\E[h(W)]-\E[h(Z)]}\\
&\leq  \frac{M_1(h)\Opnorm{\V^{-1/2}}}{p_1\sqrt{2\pi}}\sum_{i,k=1}^d(p_i+p_k)\biggl(\sum_{\substack{M\subseteq[n]:\\\abs{M}\leq p_i+p_k-1}}\Var\bigl(U_M(i,k)\bigr)\biggr)^{1/2}\\
&\;+\frac{\sqrt{2\pi d}}{6p_1}M_2(h)\Opnorm{\V^{-1/2}}\sum_{i=1}^d p_i\sigma_n(i)\biggl(\sum_{\substack{M\subseteq[n]:\\\abs{M}\leq 2p_i-1}}\Var\bigl(U_M(i,i)\bigr)\biggr)^{1/2}\\
&\;+\frac{\sqrt{\pi d}}{6p_1\sqrt{n}}M_2(h)\Opnorm{\V^{-1/2}}\sum_{i=1}^d p_i^{3/2}\sigma_n(i)^3\sqrt{\kappa_{p_i}}\,.
\end{align*}
\end{enumerate}
\end{lemma}

We next state our main multivariate normal approximation theorem, { and some more notation is needed for the sake of readability}. For $1\leq i,k\leq d$, we define 

\begin{align*}
A_1(i,k,n)&:= \sum_{t=1}^{p_i+p_k-1}\sum_{r=\ceil{\frac{t}{2}}}^{t\wedge p_i\wedge p_k} C(p_i,p_k,t,r)
\|\psi^{(i)}\star_r^{t-r}\psi^{(k)}\|_{L^2(\mu^{\otimes p_i+p_k-t})}\;n^{t/2-r}\\
&=\sum_{r=1}^{p_i\wedge p_k}\sum_{l=0}^{(p_i+p_k-r-1)\wedge r}C(p_i,p_k,l+r,r)\|\psi^{(i)}\star_r^{l}\psi^{(k)}\|_{L^2(\mu^{\otimes p_i+p_k-r-l})}\;n^{\frac{l-r}{2}}
\end{align*}
as well as
\begin{align*}
&A_2(i,k,n):=\sum_{s=1}^{\ceil{\frac{p_i+p_k}{2}}-1}\biggl(C(p_i,p_k,2s,s)
\|\psi^{(i)}\star_s^s\psi^{(k)}\|_{L^2(\mu^{\otimes p_i+p_k-2s})}\\
&\;+\|\psi^{(i)}\|_{L^4(\mu^{\otimes p_i})}\|\psi^{(k)}\|_{L^4(\mu^{\otimes p_k})}\sum_{r=s+1}^{(2s)\wedge p_i\wedge p_k} 
C(p_i,p_k,2s,r)\;n^{s-r}\biggr)\\
&\;+\|\psi^{(i)}\|_{L^4(\mu^{\otimes p_i})}\|\psi^{(k)}\|_{L^4(\mu^{\otimes p_k})}\sum_{s=1}^{\floor{\frac{p_i+p_k}{2}}}\sum_{r=s}^{(2s-1)\wedge p_i\wedge p_k} C(p_i,p_k,2s-1,r)\;n^{s-r-1/2}\,,
\end{align*}
where the constants $C(p,q,t,r)$ have been defined in Lemma \ref{tele2}. { As indicated in the statement below, each of the two estimates appearing in Theorem \ref{mdimbound} hold when either $A_1$ or $A_2$ is plugged on the right-hand side -- the key to this phenomenon being the subsequent Lemma \ref{preplemma}}.

\begin{theorem}\label{mdimbound}
With the above notation and assumptions, { the following estimates hold}.
\begin{enumerate}[{\normalfont (i)}]
\item For any $h\in C^3(\R^d)$ such that $\E\bigl[\abs{h(W)}\bigr]<\infty$ and $\E\bigl[\abs{h(Z)}\bigr]<\infty$ and for $j=1,2$ we have 
\begin{align*}
\babs{\E[h(W)]-\E[h(Z)]}&\leq \frac{1}{4p_1}\tilde{M}_2(h)\sum_{i,k=1}^d(p_i+p_k) A_j(i,k,n) \\
&\;+\frac{2M_3(h)\sqrt{d}}{9p_1}\sum_{i=1}^d p_i\|\psi^{(n,i)}\|_{L^2(\mu^{\otimes p})} A_j(i,i,n)\\
&\;+\frac{\sqrt{2d}M_3(h)}{9p_1\sqrt{n}}\sum_{i=1}^d p_i^{3/2}\|\psi^{(n,i)}\|^3_{L^2(\mu^{\otimes p})}\sqrt{\kappa_{p_i}}\,.
\end{align*}
\item If, moreover, $\V$ is positive definite, then for each $h\in C^2(\R^d)$ such that\\ $\E\bigl[\abs{h(W)}\bigr]<\infty$ and $\E\bigl[\abs{h(Z)}\bigr]<\infty$  and for $j=1,2$ we have 
\begin{align*}
\babs{\E[h(W)]-\E[h(Z)]}&\leq  \frac{M_1(h)\Opnorm{\V^{-1/2}}}{p_1\sqrt{2\pi}}\sum_{i,k=1}^d(p_i+p_k)A_j(i,k,n)\\
&\;+\frac{\sqrt{2\pi d}}{6p_1}M_2(h)\Opnorm{\V^{-1/2}}\sum_{i=1}^d p_i\|\psi^{(n,i)}\|_{L^2(\mu^{\otimes p})} A_j(i,i,n)\\
&\;+\frac{\sqrt{\pi d}}{6p_1\sqrt{n}}M_2(h)\Opnorm{\V^{-1/2}}\sum_{i=1}^d p_i^{3/2}\|\psi^{(n,i)}\|^3_{L^2(\mu^{\otimes p})}\sqrt{\kappa_{p_i}}\,.
\end{align*}
\end{enumerate}
\end{theorem}

For the proof of Theorem \ref{mdimbound} we will need the following preparatory result.

\begin{lemma}\label{preplemma}
For all $1\leq i,k\leq d$, one has the estimates
\begin{align}\label{e:twoest}
 \biggl(\sum_{\substack{M\subseteq[n]:\\\abs{M}\leq p_i+p_k-1}}\Var\bigl(U_M(i,k)\bigr)\biggr)^{1/2}&\leq A_1(i,k,n)\leq A_2(i,k,n)\,.
\end{align}
\end{lemma}

\begin{proof}[Proof of Lemma \ref{preplemma}]
By orthogonality and the product formula { stated in} Proposition \ref{pform}, we have 
\begin{align}\label{pl1}
\sum_{\substack{M\subseteq[n]:\\\abs{M}\leq p_i+p_k-1}}\Var\bigl(U_M(i,k)\bigr)
&=\sum_{t=1}^{p_i+p_k-1}\Var\Bigl(\sum_{\substack{M\subseteq[n]:\\\abs{M}= p_i+p_k-t}}U_M(i,k)\Bigr)\notag\\
&=\sum_{t=1}^{p_i+p_k-1}\Var\Bigl(J_{p_i+p_k-t}\bigl(\chi_{p_i+p_k-t}\bigr)\Bigr)\,,
\end{align}
where 
\begin{align*}
 \chi_{p_i+p_k-t}&=\sum_{r=\ceil{\frac{t}{2}}}^{t\wedge p_i\wedge p_k}\binom{n-p_i-p_k+t}{t-r}\binom{p_i+p_k-t}
{p_i-r,p_k-r,2r-t}\bigl(\widetilde{\phi^{(i)}\star_r^{t-r}\phi^{(k)}}\bigr)_{p_i+p_k-t}\\
&= \sum_{r=\ceil{\frac{t}{2}}}^{t\wedge p_i\wedge p_k} \frac{\binom{n-p_i-p_k+t}{t-r} }{\sqrt{\binom{n}{p_i}}\sqrt{\binom{n}{p_k}}}  \binom{p_i+p_k-t}
{p_i-r,p_k-r,2r-t}   \bigl(\widetilde{\psi^{(i)}\star_r^{t-r}\psi^{(k)}}\bigr)_{p_i+p_k-t}\,.
\end{align*}
By Lemmas \ref{tele1} and \ref{tele2}, arguing similarly as in the proof of Theorem \ref{1dimbound}, we obtain for $1\leq t\leq p_i+p_k-1$ that 
\begin{align}\label{pl2}
&\sqrt{\Var\Bigl(J_{p_i+p_k-t}\bigl(\chi_{p_i+p_k-t}\bigr)\Bigr)}=\sqrt{\binom{n}{p_i+p_k-t}}
\|\chi_{p_i+p_k-t}\|_{L^2(\mu^{\otimes p_i+p_k-t})}\notag\\
&\leq \frac{\sqrt{\binom{n}{p_i+p_k-t}}}{\sqrt{\binom{n}{p_i}}\sqrt{\binom{n}{p_k}}}
\sum_{r=\ceil{\frac{t}{2}}}^{t\wedge p_i\wedge p_k}\binom{n-p_i-p_k+t}{t-r}\binom{p_i+p_k-t}
{p_i-r,p_k-r,2r-t}\notag\\
&\hspace{4cm}\cdot\|\psi^{(i)}\star_r^{t-r}\psi^{(k)}\|_{L^2(\mu^{\otimes p_i+p_k-t})}\notag\\
&\leq \sum_{r=\ceil{\frac{t}{2}}}^{t\wedge p_i\wedge p_k} C(p_i,p_k,t,r)
\|\psi^{(i)}\star_r^{t-r}\psi^{(k)}\|_{L^2(\mu^{\otimes p_i+p_k-t})}\;n^{t/2-r}\,.
\end{align}
This proves the first inequality { in \eqref{e:twoest}}. The second { estimate in \eqref{e:twoest}} can be deduced from the first one by again distinguishing the cases 
of even and odd values of $1\leq t\leq p_i+p_k-1$ and by using the statement of Lemma \ref{contrlemma}-(iv) 
in the cases $t/2\not= r$. 
\end{proof}

\begin{proof}[Proof of Theorem \ref{mdimbound}]
The theorem follows immediately from Lemmas \ref{mdimlemma} and \ref{preplemma}.
\end{proof}

\section{Bounds for General Symmetric $U$-statistics}\label{general}
{ As anticipated, we now want to apply the multidimensional results of the previous section in order to deal with the one-dimensional normal approximation of general $U$-statistics; in particular, our main aim is to develop tools for systematically dealing with sequences of $U$-statistics {\it without a dominant Hoeffding component} -- thus falling in principle outside the scope of Section \ref{genU}. As before $X = (X_1,...,X_n)$, $n\geq 1$, indicates a vector of i.i.d. random variables, with values in $(E, \mathcal{E})$, and common distribution $\mu$}. 

\medskip

We let $\psi:E^p\rightarrow\R$ be a symmetric kernel of order $p$ which is neither necessarily degenerate nor has a dominating component. 
From \eqref{HDsym} we know that the random variable 
$F:=J_p(\psi)$ has the Hoeffding decomposition 
\begin{equation*}
 F=\sum_{s=0}^p \binom{n-s}{p-s} J_s(\psi_s)=\E[F]+\sum_{s=1}^p \binom{n-s}{p-s} J_s(\psi_s)\,,
\end{equation*}
where the symmetric and degenerate kernels $\psi_s:E^s\rightarrow\R$ of order $s$ are given by \eqref{defpsis}. We will assume that 
\[0<\sigma^2:=\Var(J_p(\psi))=\sum_{s=1}^p\binom{n-s}{p-s}^2\binom{n}{s}\|\psi_s\|_{L^2(\mu^{\otimes s})}^2<+\infty\]
and write
\begin{equation*}
 W:=\frac{F-\E[F]}{\sqrt{\Var(F)}}
\end{equation*}
for the normalised version of $F$. Our goal is to use the multivariate bounds from Theorem \ref{mdimbound} in order to estimate a suitable distance of $W$ to a standard normal random variable $Z { \sim N(0,1)}$. Note that the Hoeffding decomposition 
of $W$ is given by 
\begin{equation}\label{HDFtilde}
 W=\sum_{s=1}^p J_s\biggl(\frac{\binom{n-s}{p-s}}{\sigma}\,\psi_s\biggr)=\sum_{s=1}^p J_s\bigl(\phi^{(s)}\bigr)\,,
\end{equation}
where, in accordance with the notation from Section \ref{ss:multisetup}, we define
\begin{align*}
 \phi^{(s)}&=\phi^{(n,s)}:=\frac{\binom{n-s}{p-s}}{\sigma}\,\psi_s\quad\text{and}\\
 \psi^{(s)}&={ \psi^{(n,s)}}:=\sqrt{\binom{n}{s}}\phi^{(s)}=\frac{\sqrt{\binom{n}{s}}\binom{n-s}{p-s}}{\sigma}\,\psi_s\,,\quad 1\leq s\leq p\,.
\end{align*}
Note that, by construction, we have 
\begin{align}\label{normalization}
 1&=\Var\bigl(W\bigr)=\sum_{s=1}^p\Var\bigl(J_s(\phi^{(s)})\bigr)=\sum_{s=1}^p\frac{\binom{n}{s}\binom{n-s}{p-s}^2}{\sigma^2}{ \|\psi_s\|_{L^2(\mu^{\otimes s})}^2}=\sum_{s=1}^p \|\psi^{(s)}\|^2_{L^2(\mu^{\otimes s})}\,,
\end{align}
which implies that 
\begin{equation}\label{boundL2}
 0\leq  \|\psi^{(s)}\|_{L^2(\mu^{\otimes s})}\leq 1\,,\quad 1\leq s\leq p\,.
\end{equation}

In order to apply Theorem \ref{mdimbound}, we must estimate the following contraction norms.
\begin{align}\label{genconts}
 \|\psi^{(i)}\star_s^l\psi^{(k)}\|_{L^2(\mu^{\otimes i+k-s-l})}&=\frac{\sqrt{\binom{n}{i}}\binom{n-i}{p-i}\sqrt{\binom{n}{k}}\binom{n-k}{p-k}}{\sigma^2}\|\psi_{i}\star_s^l\psi_{k}\|_{L^2(\mu^{\otimes i+k-s-l})}\,,
\end{align}
where $0\leq i,k\leq p$, $1\leq s\leq i\wedge k$ and $0\leq l\leq (i+k-s-1)\wedge s$. 
Since the kernels $\psi_i$, $1\leq i\leq p$, appearing in \eqref{defpsis} have complicated expressions and are, hence, not straightforward to compute in practice, we provide the following lemma, taken from \cite{DKP}, which bounds these norms in terms of 
norms of contractions of the { (much)} simpler functions $g_k$ given by \eqref{gk}. For the reader's convenience, the proof is included in Section \ref{proofs}.
In order to state it we introduce the following notation. 
For positive integers $1\leq r, i, k \leq p$ and $0\leq l\leq p$ such that $0\leq l\leq r\leq i\wedge k$ let 
$Q(i,k,r,l)$ be the set of quadruples $(j,m,a,b)$ of nonnegative integers such that the following hold:
\begin{enumerate}[(1)]
\item $j\leq i$ and $m\leq k$.
\item $b\leq a\leq r$.
\item $b\leq l$.
\item $a-b\leq r-l$.
\item $j+m-a-b\leq i+k-r-l\leq i+k-1$.
\item $a\leq j\wedge m$.
\item If $j=m=p$, then $b=l$ and $a=r\geq1$.
\end{enumerate}

\begin{lemma}[Lemma 5.7 in \cite{DKP}] \label{genulemma}
With the above notation, for positive integers $1\leq r, i, k \leq p$ and $0\leq l\leq p$ such that $0\leq l\leq r\leq i\wedge k$ there exists a constant $K(i,k,r,l)\in(0,\infty)$ only depending on $i,k,r$ and $l$ such that 
\begin{align*}
 \|\psi_{i}\star_r^l\psi_{k}\|_{L^2(\mu^{\otimes i+k-r-l})}&\leq K(i,k,r,l) \max_{(j,m,a,b)\in Q(i,k,r,l)}\|g_j\star_a^b g_m\|_{L^2(\mu^{\otimes j+m-a-b})}\,.
\end{align*}
\end{lemma}

In order to estimate the quantities $A_2(i,k,n)$ from Theorem \ref{mdimbound}, we still have to bound the $L^4$-norms $\|\psi^{(i)}\|_{L^4(\mu^{\otimes s})}$ for $1\leq i\leq p$.
Since
\begin{align*}
 \|\psi_{i}\|_{L^4(\mu^{\otimes i})}^2&=\|\psi_{i}\star_i^0\psi_{i}\|_{L^2(\mu^{\otimes i})}\,,
\end{align*}
we obtain from Lemma \ref{genulemma} that 
\begin{align}\label{boundl4}
 \|\psi^{(i)}\|_{L^4(\mu^{\otimes i})}^2&=\frac{\binom{n}{i}\binom{n-i}{p-i}^2}{\sigma^2}\|\psi_{i}\|_{L^4(\mu^{\otimes i})}^2\notag\\
 &\leq \frac{\binom{n}{i}\binom{n-i}{p-i}^2}{\sigma^2} K(i,i,i,0)\max_{0\leq j,m,a\leq i}\|g_j\star_a^0 g_m\|_{L^2(\mu^{\otimes j+m-a})}\,.
\end{align}

In order to state our normal approximation result for $W$, let us introduce the following notation. For $1\leq i,k\leq p\leq n$ define  
\begin{align*}
 B_1(i,k,n)&:=\sum_{s=1}^{i\wedge k}\sum_{l=0}^{(i+k-s-1)\wedge s} C(i,k,l+s,s) K(i,k,s,l)\\
 & \hspace{1cm}\cdot n^{\frac{l-s}{2}}\frac{\sqrt{\binom{n}{i}}\binom{n-i}{p-i}\sqrt{\binom{n}{k}}\binom{n-k}{p-k}}{\sigma^2}
\max_{(j,m,a,b)\in Q(i,k,s,l)}\|g_j\star_a^b g_m\|_{L^2(\mu^{\otimes j+m-a-b})}
\end{align*}
as well as 
\begin{align*}
 &B_2(i,k,n):=\sum_{s=1}^{\ceil{\frac{i+k}{2}}-1}\biggl(C(i,k,2s,s) K(i,k,s,s) \frac{\sqrt{\binom{n}{i}}\binom{n-i}{p-i}\sqrt{\binom{n}{k}}\binom{n-k}{p-k}}{\sigma^2}\\
 &\hspace{3cm}\cdot  \max_{(j,m,a,b)\in Q(i,k,s,s)}\|g_j\star_a^b g_m\|_{L^2(\mu^{\otimes j+m-a-b})}\\
&\;+\frac{\sqrt{\binom{n}{i}}\binom{n-i}{p-i}\sqrt{\binom{n}{k}}\binom{n-k}{p-k}}{\sigma^2}\Bigl(K(k,k,k,0)\max_{(j,m,a,0)\in Q(k,k,k,0)}\|g_j\star_a^0 g_m\|_{L^2(\mu^{\otimes j+m-a})}\\
&\hspace{1cm}K(i,i,i,0) \max_{(j,m,a,0)\in Q(i,i,i,0)}\|g_j\star_a^0 g_m\|_{L^2(\mu^{\otimes j+m-a})}\Bigr)^{1/2}\cdot\sum_{r=s+1}^{(2s-1)\wedge i\wedge k}C(i,k,2s,r)\;n^{s-r}\biggr)\\
&\;+\frac{\sqrt{\binom{n}{i}}\binom{n-i}{p-i}\sqrt{\binom{n}{k}}\binom{n-k}{p-k}}{\sigma^2}\Bigl(K(k,k,k,0)\max_{(j,m,a,0)\in Q(k,k,k,0)}\|g_j\star_a^0 g_m\|_{L^2(\mu^{\otimes j+m-a})}\\
&\hspace{1cm}K(i,i,i,0) \max_{(j,m,a,0)\in Q(i,i,i,0)}\|g_j\star_a^0 g_m\|_{L^2(\mu^{\otimes j+m-a})}\Bigr)^{1/2}\\
&\hspace{2cm}\cdot\sum_{s=1}^{\floor{\frac{i+k}{2}}}\sum_{r=s}^{(2s-1)\wedge i\wedge k} C(i,k,2s-1,r)\;n^{s-r-1/2}\,,
\end{align*}
where the constants $C(i,k,t,s)$ and $K(i,k,s,l)$ are those from Lemmas \ref{tele2} and \ref{genulemma}, respectively. { Despite their complicated definition, dealing with bounds involving $B_1$ and $B_2$ is actually rather straightforward, once one observes }that there are finite constants $b_1(i,k)$ and $b_2(i,k)$ such that 
\begin{align*}
 B_1(i,k,n)&\leq b_1(i,k)\max_{\substack{1\leq s\leq i\wedge k,\\0\leq l\leq(i+k-s-1)\wedge s}}
\frac{n^{2p-(i+k+s-l)/2}}{\sigma^2}\\
&\hspace{2cm}\cdot\max_{(j,m,a,b)\in Q(i,k,s,l)}\|g_j\star_a^b g_m\|_{L^2(\mu^{\otimes j+m-a-b})}\\
&=:B_1'(i,k,n)
\end{align*}
and
\begin{align*}
 B_2(i,k,n)&\leq b_2(i,k) \1_{(i+k>2)} \frac{n^{2p-(i+k)/2}}{\sigma^2}\max_{\substack{1\leq s\leq \ceil{\frac{i+k}{2}}-1\\ (j,m,a,b)\in Q(i,k,s,s)}}\|g_j\star_a^b g_m\|_{L^2(\mu^{\otimes j+m-a-b})}  \\
 &\;+b_2(i,k)\frac{n^{2p-(i+k+1)/2}}{\sigma^2}  \Bigl( \max_{ (j,m,a,0)\in Q(i,i,i,0)}\|g_j\star_a^0 g_m\|_{L^2(\mu^{\otimes j+m-a})}\\
 &\hspace{2cm}\cdot\max_{(j,m,a,0)\in Q(k,k,k,0)}\|g_j\star_a^0 g_m\|_{L^2(\mu^{\otimes j+m-a})}\Bigr)^{1/2} \\
&=:B_2'(i,k,n)\,.
\end{align*}

\begin{theorem}[Normal approximation of general symmetric $U$-statistics]\label{genutheo}
 Let $W$ be as above and let { $N$} be a standard normal random variable. Furthermore, let $g\in C^3(\R)$ have three bounded 
derivatives. Then, for $j=1,2$, we have the bound 
\begin{align}
\babs{\E[g(W)]-\E[g(N)]}&\leq \frac{1}{4}\sqrt{p}\fnorm{g''} \sum_{i,k=1}^p (i+k) B_j(i,k,n) \notag \\
&\;+\frac{2 \fnorm{g'''} \sqrt{p}}{9}\sum_{i=1}^p i\|\psi^{(n,i)}\|_{L^2(\mu^{\otimes p})} B_j(i,i,n)\label{e:tower}\\
&\;+\frac{\sqrt{2p}\fnorm{g'''}}{9\sqrt{n}}\sum_{i=1}^p i^{3/2}\|\psi^{(n,i)}\|^3_{L^2(\mu^{\otimes p})}\sqrt{\kappa_{i}}\notag
\end{align}
{ and an analogous inequality holds} with the constants $B_j(i,k,n)$ replaced by the respective $B_j'(i,k,n)$. Here, $\kappa_i$ is a finite constant depending only on $i$.

\end{theorem}

\begin{remark}\label{remgenu}
 \begin{enumerate}[(a)]
  \item Note that, by using \eqref{boundL2}, the bound in Theorem \ref{genutheo} could further be simplified but we prefered leaving it as it is since there might be cases where it is possible to 
 estimate the quantities $\|\psi^{(n,i)}\|_{L^2(\mu^{\otimes p})}$ more accurately. 
  \item {A drawback of our approach is that} Theorem \ref{genutheo} allows one to only bound expressions involving $C^3$ test functions. {Such a technical limitation is an artifact of our method of proof, involving a} detour through the multivariate normal approximation result stated in Theorem \ref{mdimbound}. {On the other hand, our derivation of \eqref{e:tower} from a multidimensional result immediately implies that, if one can prove that the right-hand side of \eqref{e:tower} converges to zero as $n\to \infty$, then one can immediately deduce the joint convergence of the vector of Hoeffding components of the $U$-statistic $W$ to some multivariate normal distribution. From a qualitative point of view, this seems to be a much stronger statement than that the simple convergence of { $W$}, since the latter might a priori be due to certain cancellation effects.} Observe that, as several Hoeffding components of { $W$} might vanish in the limit (thus generating a singular covariance matrix), in the proof of \eqref{e:tower} we can only invoke part (i) of Theorem \ref{mdimbound} which gives a bound in terms of $C^3$ test functions. {In general, recurring to smoother test functions seems to be inevitable when using Stein's method for} multivariate normal approximation, when one does not deal with { an invertible}
  limiting covariance matrix. { As already discussed}, whenever one Hoeffding component is dominant, then one might use the bound from Theorem \ref{1dimth2} in order to obtain a bound on the Wasserstein distance. 
  \item We stress that our bound is purely analytic and that the functions $g_k$, whose contraction norms must be evaluated, are typically much easier to compute than the individual Hoeffding kernels $\psi_s$ which are alternating 
  sums of the $g_k$ for $0\leq k\leq s$ (see \eqref{defpsis}). Apart from these norms, the only quantity which has to be controlled is the variance $\sigma^2$ of $F$. 
  \item We remark that the maxima appearing in the definition of the quantities $B_j(i,k,n)$ and $B_j'(i,k,n)$ give certain important constraints on the indices $s,l,j,m,a$ and $b$. This is comparable to similar constraints 
  appearing in the bounds provided in \cite{LRP1} and \cite{LRP2}. In particular, when dealing with example cases, it is usually important to take these constraints into account in order to show that the bounds indeed converge to zero. 
  This is for instance the case in the example dealt with in Section \ref{apps}.
	{\item Using a linear projection $\R^{p_1+\ldots+p_d}\rightarrow\R^d$, we could similarly use Theorem \ref{mdimbound} in order to provide a bound on the $d$-dimensional normal approximation of a vector of non-degenerate $U$-statistics of respective orders $p_1,\dotsc,p_d$. This is clear from the proof of Theorem \ref{genutheo}. }
 \end{enumerate}

\end{remark}

\begin{proof}[Proof of Theorem \ref{genutheo}]
 Let $g\in C^3(\R)$ have three bounded derivatives. Define $S:\R^p\rightarrow\R$ by $S(x_1,\dotsc,x_p):=\sum_{j=1}^p x_j$ as well as $h:\R^p\rightarrow\R$ by $h:=g\circ S$. Then, $h\in C^3(\R^p)$,  and one can easily check 
 that 
 \[\frac{\partial h^k}{\partial x_{i_1}\ldots\partial x_{i_k}}= g^{(k)}\circ S\,,\quad 0\leq k\leq3\,.\] 
 In particular, it follows that 
 \[M_k(h)=\fnorm{g^{(k)}}\,, \quad 0\leq k\leq 3\quad\text{as well as  } \tilde{M}_2(h)\leq \sqrt{p}\fnorm{g''}\,,\]
 where the last inequality is by \eqref{M2bound}. Let $\mathbb{V}$ be the covariance matrix of the vector 
 \[V:=\bigl(J_1(\phi^{(1)}),\dotsc,J_p(\phi^{(p)})\bigr)^T\,.\]
 Then, $S(V)=W$, $\mathbb{V}$ is diagonal and by \eqref{normalization} its diagonal entries sum up to $1$. Hence, letting $Z=(Z_1,\dotsc,Z_p)^T$ be a centered $p$-dimensional normal vector with covariance matrix $\mathbb{V}$, it follows that 
 $S(Z)$ has the standard normal distribution of $N$. It is easy to see that plugging in the bounds on the contractions $\|\psi_{i}\star_s^l\psi_{k}\|_{L^2(\mu^{\otimes i+k-s-l})}$ provided by Lemma \ref{genulemma} and \eqref{boundl4} as well as 
 respecting \eqref{genconts} yields the bounds $B_j(i,k,n)$ which are themselves bounded from above by the $B_j'(i,k,n)$. Finally, we notice that 
 \begin{align*}
  \babs{\E[g(W)]-\E[g(N)]}&=\babs{\E[h(V)]-\E[h(Z)]}\,,
 \end{align*}
for which an upper bound is provided in Theorem \ref{mdimbound} (i). 
\end{proof}

\section{An Application to Subgraph Counting}\label{apps}

{\it Geometric random graphs} are graphs whose vertices are random points scattered on some Euclidean domain, and whose edges are determined by some explicit geometric rule; in view of their wide applicability (for instance, to the modelling of telecommunication networks), these objects represent a very popular and important alternative to the combinatorial Erd\"os-R\'enyi random graphs. We refer to the monographs \cite{Penrose} and \cite{PecRei16} for a {detailed} introduction to this topic and its several applications. We will use our Theorem \ref{genutheo} in order to prove the Gaussian fluctuations of subgraph counts in a typical model of this kind. Although the asymptotic (jointly) Gaussian behaviour of these counts is well understood both in the 
binomial and in the Poisson point process situation (at least at the qualitative level, see again \cite{Penrose}), we chose this example in order to demonstrate the power and easy applicability of our bounds. As already discussed, in the case of uniformly distributed points on some Euclidean domain, our results yield a substantial refinement and extension of \cite{BhaGo92, JJ}.  {In the case where the vertices of the random graph are generated by a Poisson measure}, the recent paper \cite{LRP1} provides the univariate CLT with a rate of convergence for the Wasserstein distance. 

\medskip

We fix a dimension $d\geq1$ as well as a bounded and Lebesgue almost everywhere continuous probability density function $f$ on $\R^d$. Let $\mu(dx):=f(x)dx$ be the corresponding probability measure on $(\R^d,\B(\R^d))$ and suppose that 
$X_1,X_2,\dotsc$ are i.i.d. with distribution $\mu$. Let $X:=(X_j)_{j\in\N}$. We denote by $(t_n)_{n\in\N}$ a sequence of radii in $(0,\infty)$ such that $\lim_{n\to\infty}t_n=0$. For each $n\in\N$, we denote by $G(X;t_n)$ the 
\textbf{random geometric graph} obtained as follows. The vertices {of $G(X;t_n)$} are given by the set $V_n:=\{X_1,\dotsc,X_n\}$, which $\Prob$-a.s. has cardinality $n$, 
and two vertices $X_i,X_j$ are connected if and only if $0<\Enorm{X_i-X_j}<t_n$. Furthermore, let $p\geq2$ be a fixed integer and suppose that 
$\Gamma$ is a fixed connected graph on $p$ vertices. For each $n$ we denote by $G_n(\Gamma)$ the number of induced subgraphs of $G(X;t_n)$ which are isomorphic to $\Gamma$. 
Recall that an induced subgraph of $G(X;t_n)$ consists of a non-empty subset $V_n'\subseteq V_n$ and its edge set is precisely the set of edges of $G(X;t_n)$ whose endpoints are {both} in $V_n'$. 
We will also have to assume that $\Gamma$ is \textit{feasible} for every 
$n\geq p$. This means that the probability that the restriction of $G(X;t_n)$ to $X_1,\dotsc,X_p$ is isomorphic to $\Gamma$ is strictly positive for $n\geq p$. Note that feasibility depends on the common distribution $\mu$ of the points.
The quantity $G_n(\Gamma)$ is a symmetric $U$-statistic of $X_1,\dotsc,X_n$ since 
\begin{equation*}
 G_n(\Gamma)=\sum_{1\leq i_1<\ldots<i_p\leq n} \psi_{\Gamma,t_n}(X_{i_1},\dotsc,X_{i_p})\,,
\end{equation*}
where $\psi_{\Gamma,t_n}(x_1,\dotsc,x_p)$ equals $1$ if the graph with vertices $x_1,\dotsc,x_p$ and edge set $\{\{x_i,x_j\}\,:\, 0<\Enorm{x_i-x_j}< t_n\}$ is isomorphic to $\Gamma$ and $0$, otherwise. For obtaining asymptotic normality one typically distinguishes between three different asymptotic regimes (see Remark \ref{escusationonpetita} (b) below):
\begin{enumerate}
 \item[\textbf{(R1)}] $nt_n^d\to0$ and $n^p t_n^{d(p-1)}\to\infty$ as $n\to\infty$ (\textit{sparse regime})
 \item[\textbf{(R2)}] $n t_n^d\to\infty$ as $n\to\infty$ (\textit{dense regime})
 \item[\textbf{(R3)}] $nt_n^d\to\rho\in(0,\infty)$ as $n\to\infty$ (\textit{thermodynamic regime})
\end{enumerate}
It turns out that, under regime \textbf{(R2)} one also has to take into account whether the common distribution $\mu$ of the $X_j$ is the uniform distribution $\mathcal{U}(M)$ on some Borel subset $M\subseteq\R^d$, $0<\lambda^d(M) <\infty$ with 
density $f(x)=\lambda^d(M)^{-1}\,\1_M(x)$, or not. 
To take into account this specific situation, we will therefore distinguish between the following four cases:
\begin{enumerate}
 \item[\textbf{(C1)}] $nt_n^d\to0$ and $n^p t_n^{d(p-1)}\to\infty$ as $n\to\infty$. 
 \item[\textbf{(C2)}] $n t_n^d\to\infty$ as $n\to\infty$ and $\mu= \mathcal{U}(M)$ for some Borel subset $M\subseteq\R^d$ s.t. $0<\lambda^d(M) <\infty$.
 \item[\textbf{(C3)}] $n t_n^d\to\infty$ as $n\to\infty$, and $\mu$ is not a uniform distribution.
 \item[\textbf{(C4)}] $nt_n^d\to\rho\in(0,\infty)$ as $n\to\infty$.
 \end{enumerate}

The following important variance estimates will be needed (in what follows, for $a_n,b_n>0$, $n\in\N$, we write $a_n\sim b_n$ if $\lim_{n\to\infty} a_n/b_n=1$).
\begin{proposition}\label{regprop}
 Under all regimes {\normalfont \textbf{(R1)}}, {\normalfont\textbf{(R2)}} and {\normalfont\textbf{(R3)}} it holds that\\
 $\E[G_n(\Gamma)]\sim c n^{p}t_n^{d(p-1)}$ for a constant $c\in(0,\infty)$. Moreover, there exist constants $c_1,c_2,c_3,c_4\in(0,\infty)$ such that, as $n\to\infty$,  
 \begin{enumerate}
  \item[{\normalfont \textbf{(C1)}}] $\Var(G_n(\Gamma))\sim c_1\cdot n^p t_n^{d(p-1)}$,
  \item[{\normalfont \textbf{(C2)}}] $\Var(G_n(\Gamma))\geq c_2\cdot n^{p} t_n^{d(p-1)}$ for all $n\in\N$,
  \item[{\normalfont \textbf{(C3)}}] $\Var(G_n(\Gamma))\sim c_3\cdot n^{2p-1} t_n^{d(2p-2)}$,
  \item[{\normalfont \textbf{(C4)}}] $\Var(G_n(\Gamma))\sim c_4\cdot n$.
 \end{enumerate}
\end{proposition}

\begin{proof}
 The formulas on the asymptotic variances given in Theorems 3.12 and 3.13 in the book \cite{Penrose} yield the claims in the cases {\normalfont\textbf{(C1)}} and {\normalfont\textbf{(C3)}} and {\normalfont\textbf{(C4)}}. 
 However, in the case {\normalfont\textbf{(C2)}}, the limiting covariance appearing in \cite[Theorem 3.12]{Penrose} is actually equal to zero, from which one can only infer that the actual order of the variance of $G_n(\Gamma)$ is of a smaller order than $n^{2p-1} t_n^{d(2p-2)}$. In order to compute an effective lower bound for such a variance, we will apply formula \eqref{varlb2}. 
 Indeed, by \eqref{varlb2} we have 
 \begin{align}\label{vb1}
  \Var\bigl(G_n(\Gamma)\bigr)&\geq \binom{n}{p}\Var\bigl(\psi_{\Gamma, t_n}(X_1,\dotsc,X_p)\bigr)\notag\\
  &=\binom{n}{p}\E\bigl[\psi_{\Gamma, t_n}(X_1,\dotsc,X_p)\bigr]-\binom{n}{p}\Bigl(\E\bigl[\psi_{\Gamma, t_n}(X_1,\dotsc,X_p)\bigr]\Bigr)^2\notag\\
  &= \E\bigl[G_n(\Gamma)\bigr]- \binom{n}{p}^{-1}\Bigl(\E\bigl[G_n(\Gamma)\bigr] \Bigr)^2\,,
 \end{align}
where we have used the fact that $\psi_{\Gamma,t_n}^2=\psi_{\Gamma,t_n}$ for the second identity. Now, from \cite[Proposition 3.1]{Penrose} we know that 
\begin{align*}
 \E\bigl[G_n(\Gamma)\bigr]&\sim n^{p}t_n^{d(p-1)}\mu_\Gamma\,,
\end{align*}
where 
\begin{align*}
 \mu_\Gamma= (p!)^{-1}\int_{\R^d} f(x)^pdx\int_{(\R^d)^{p-1}}\psi_{\Gamma,1}(0,y_2,\dotsc,y_p)dy_2\ldots dy_p>0\,.
\end{align*}
Hence, 
\begin{align*}
 \binom{n}{p}^{-1}\Bigl(\E\bigl[G_n(\Gamma)\bigr] \Bigr)^2&\sim p!\mu_\Gamma^2 n^p t_n^{2d(p-1)}= o\Bigl(\E\bigl[G_n(\Gamma)\bigr]\Bigr)
\end{align*}
and we obtain from \eqref{vb1} that indeed 
\begin{equation*}
 \Var\bigl(G_n(\Gamma)\bigr)\geq c_2 n^{p}t_n^{d(p-1)}\,,\quad n\in\N\,,
\end{equation*}
for a positive constant $c_2$.
\end{proof}

We denote by 
\begin{equation*}
 W:=W_n:=\frac{G_n(\Gamma)-\E[G_n(\Gamma)]}{\sqrt{\Var(G_n(\Gamma))}}
\end{equation*}
the normalized version of $G_n(\Gamma)$. {The following statement is a direct application of the main results of this paper.}

\begin{theorem}\label{subcounts}
 Let $N$ be a standard normal random variable. Then, with the above definitions and notation, for every function $g\in C^3(\R)$ with three bounded derivatives, there exists a finite constant $C>0$ which is 
 independent of $n$ such that for all $n\geq p$,
 \begin{align*}
  \babs{\E[g(W)]-\E[g(N)]}&\leq C\cdot\bigl(n^p t_n^{d(p-1)}\bigr)^{-1/2}\quad\text{in case {\normalfont\textbf{(C1)}}}\,,\\
  \babs{\E[g(W)]-\E[g(N)]}&\leq C\cdot n^{-1/2} \quad\text{in cases {\normalfont\textbf{(C3)}} and {\normalfont\textbf{(C4)} and}}\\
  \babs{\E[g(W)]-\E[g(N)]}&\leq C\cdot \bigl(n^{2p-3}t_n^{d(2p-2)}\bigr)^{1/2} \quad\text{in case {\normalfont\textbf{(C2)}}}\,.
 \end{align*}
In particular, we have that $W_n$ always converges in distribution to $N$ as $n\to\infty$ in the cases {\normalfont \textbf{(C1)}}, {\normalfont \textbf{(C3)}} and {\normalfont\textbf{(C4)}}. 
In case {\normalfont\textbf{(C2)}}, we have that $W_n$ converges in distribution to $N$ under the additional assumption that $\lim_{n\to\infty} n^{2p-3}t_n^{d(2p-2)}=0$.
 \end{theorem}

\begin{remark}\label{escusationonpetita} 
\begin{enumerate}[(a)]
 \item The proof of Theorem \ref{subcounts} provided below is remarkably short -- in particular, because we are able to directly exploit several technical computations taken from \cite{LRP2}. The fact that a CLT for $U$-statistics based on i.i.d. samples can now be directly proved by slightly adapting the computations for the Poisson setting is a demonstration of the power of Theorem \ref{genutheo} above, allowing one to replace estimates involving the kernels of Hoeffding decompositions with considerably simpler expressions. As a side remark, we observe that comparable bounds could in principle be obtained by combining \cite{LRP2} with a de-Poissonization technique analogous to \cite{DynMan83}; this would however change the rates of convergence, as well as force us to deal with some complicated conditional variance estimates and provide less complete information about the fluctuations of Hoeffding projections -- see also Remark \ref{remgenu}-(b).
 \item Note that in all the three regimes \textbf{(R1)}, \textbf{(R2)} and \textbf{(R3)} considered in Theorem \ref{subcounts}, one has that $\lim_{n\to\infty} n^pt_n^{d(p-1)}=+\infty$. Indeed, it is shown in \cite[Section 3.2]{Penrose} that 
 $W_n$ converges weakly to a Poisson distribution if $\lim_{n\to\infty} n^pt_n^{d(p-1)}=\alpha\in (0,\infty)$ and to $0$ if $\lim_{n\to\infty} n^pt_n^{d(p-1)}=0$, respectively. Hence, $\lim_{n\to\infty} n^pt_n^{d(p-1)}=+\infty$ is a necessary 
 condition for the asymptotic normality of $G_n(\Gamma)$.
 \item We remark that the distinction between uniform distributions and non-uniform distribution is not necessary for the analogous problem on Poisson space considered in \cite{Penrose,LRP2}. The reason is that, in this situation, the formulae for the respective limiting variances are slightly different, see \cite[Section 3.2]{Penrose}. The phenomenon that, in the case of a uniform distribution on a set $M$, the asymptotic order of the variance is different in the dense regime \textbf{(R2)} has already been observed in \cite[Section 4]{JJ} and in \cite[Theorem 2.1,Theorem 3.1]{BhaGo92}. It is remarked on page 1357 of \cite{JJ}, in the special case $p=2$ of edge counting, that the asymptotic order of $\Var(G_n(\Gamma))$ in fact depends on the boundary structure of the set $M$. Moreover, for smooth enough boundaries, it is claimed there that $ \Var(G_n(\Gamma))\sim c n^2 t_n^d$ for some constant $c\in(0,\infty)$, whenever $t_n=o(n^{-1/(d+1)})$. Hence, there are cases where our lower bound for $\Var(G_n(\Gamma))$ in case \textbf{(C2)} given in Proposition \ref{regprop} is sharp.  
 \item Interestingly, in the case of a uniform $\mu$, the condition {\bf($\mathbf{D_3'}$)}, which is assumed in Theorem 3.1 of \cite{BhaGo92} to guarantee asymptotic normality for the number of $p$ clusters {(that is, subgraphs of $p$ vertices that are isomorphic to the complete graph)}, is exactly the same as our additional condition that
 $$\lim_{n\to\infty} n^{2p-3}t_n^{d(2p-2)}=0.$$ {We notice that \cite[Theorem 3.1]{BhaGo92} exclusively deals with the counting of $p$-clusters, whereas our findings allow one to deduce normal fluctuations for general connected graphs of $p$ vertices.}
 \item As discussed above (see the proof of Proposition \ref{regprop}), in the situation of case \textbf{(C2)}, even the qualitative CLT for $\G_n(\Gamma)$ given in Theorem \ref{subcounts} seems to be new (for instance, in this case 
 the scaling used in Theorem 3.12 of \cite{Penrose} leads to a degenerate limit). We mention that, in the very special case of edge counting ($p=2$)  considered in \cite[Section 4]{JJ}, the authors prove that asymptotic normality holds even without the additional assumption that $\lim_{n\to\infty} n t_n^{2d}=0$.
\end{enumerate}
\end{remark}

 \begin{remark}\label{sgrem}
  We further mention that, in the cases {\normalfont \textbf{(C1)}}, {\normalfont \textbf{(C3)}} and {\normalfont\textbf{(C4)}}, we obtain the same rate of convergence as the one obtained in \cite{LRP2} in the Poisson situation for the Wasserstein distance. Moreover, if $(t_n)_{n\in\N}$ is bounded away from zero, then the CLT 
  holds true due to Hoeffding's classical CLT via the projection method \cite{Hoeffding}. In this case, a combination of Theorems \ref{1dimbound} and \ref{1dimth2} yields a bound of order $n^{-1/2}$ on the Wasserstein distance. 
 \end{remark}

\begin{proof}[Proof of Theorem \ref{subcounts}]
  Denote by $\sigma_n^2:=\Var(G_n(\Gamma))$ the variance of $G_n(\Gamma)$ and let $g_{\Gamma,t_n}^{(k)}$ be the functions defined in \eqref{gk}, corresponding to the kernel $\psi_{\Gamma,t_n}$. Moreover, fix integers $1\leq k\leq i\leq p$ and $l,r$ such that 
 $1\leq r\leq k$ and $0\leq l\leq r\wedge (i+k-r-1)$. Since $t_n\to0$ as $n\to\infty$, we can assume that $0<t_n<1$ for each $n\geq p$.
The computations on pages 4196-4197 of \cite{LRP2} show that for all 
\begin{align*}
 (j,m,a,b)\in P&:=\Bigl(\{(j,m,a,b):1\leq b\leq a\leq j\leq m\text{ and } b<m\}\\
&\hspace{2cm}\cup \{(j,m,a,b): j=m=a\text{ and } b=0\}\Bigr) \cap Q(i,k,r,l)
\end{align*}
we have that 
  \begin{align}\label{sg1}
	\|g_{\Gamma,t_n}^{(j)}\star_a^b g_{\Gamma,t_n}^{(m)}\|_{L^2(\mu^{\otimes j+m-a-b})}^2  &=O\bigl(t_n^{d(4p-(j+m+a-b)-1)}\bigr)\notag\\
	&=O\bigl(t_n^{d(4p-(i+k+r-l)-1)}\bigr)\,,
  \end{align}
	where the second relation follows from $0<t_n<1$ and the inequality $j+m+a-b\leq i+k+r-l$.
 (We observe that the authors of \cite{LRP2} actually deal with the rescaled measure $n\cdot\mu$, which is why they obtain an additional power of $n$ as a prefactor). We show first that 
the estimates in \eqref{sg1} continue to hold for quadruples $(j,m,a,b)\in Q(i,k,r,l)\setminus P$. 
To this end, we first remark that we have the asymptotic relations
\begin{align}
 \norm{g_{\Gamma,t_n}^{(m)}}_{L^2(\mu^{\otimes m})}^2&\lesssim (t_n^d)^{2p-m-1}\,,\quad 1\leq m\leq p\text{ and}\label{asrelrg1}\\
 \mu^p(\psi_{\Gamma,t_n}):=\int_{(\R^d)^p} \psi_{\Gamma,t_n} d\mu^p&\lesssim (t_n^d)^{p-1}\,.\label{asrelrg2}
\end{align}
Relation \eqref{asrelrg1} follows from the computation 
\begin{align*}
 &\norm{g_{\Gamma,t_n}^{(m)}}_{L^2(\mu^{\otimes m})}^2=\int_{(\R^d)^m} g_{\Gamma,t_n}^{(m)}(x_1,\dotsc,x_m)^2\prod_{j=1}^m f(x_j)dx_j\\
 &=\int_{(\R^d)^m}\prod_{j=1}^m f(x_j)dx_j\int_{(\R^d)^{2p-2m}} \psi_{\Gamma,t_n}(x_1,\dotsc,x_p)\\
 &\hspace{3cm}\psi_{\Gamma,t_n}(x_1,\dotsc,x_m,z_{m+1},\dotsc,z_p) \prod_{l=m+1}^p f(x_l)f(z_l)dx_ldz_l\\
 &=\int_{(\R^d)^m}\prod_{j=1}^m f(x_j)dx_j\int_{(\R^d)^{2p-2m}} \psi_{\Gamma,1}\bigl(0,t_n^{-1}(x_1-x_2) \dotsc,t_n^{-1}(x_1-x_p)\bigr)\\
 &\hspace{2cm}\psi_{\Gamma,1}\bigl(0,t_n^{-1}(x_1-x_2),\dotsc,t_n^{-1}(x_1-x_m),t_n^{-1}(x_1-z_{m+1}),\dotsc,t_n^{-1}(x_1-z_{p})\bigr)\\
&\hspace{4cm} \prod_{l=m+1}^p f(x_l)f(z_l)dx_ldz_l\\
&=(t_n^d)^{2p-m-1}\int_{\R^d} f(x_1)dx_1\int_{(\R^d)^{m-1}}\prod_{j=2}^mf(x_1+t_ny_j)dy_j\\
&\hspace{2cm}\cdot\int_{\R^{2p-2m}}\prod_{l=m+1}^pf(x_1+t_nu_l)f(x_1+t_nv_l)du_ldv_l\\
&\hspace{3cm}\psi_{\Gamma,1}(0,y_2,\dotsc,y_m,u_{m+1},\dotsc,u_p)\psi_{\Gamma,1}(0,y_2,\dotsc,y_m,v_{m+1},\dotsc,v_p)\\
&\sim (t_n^d)^{2p-m-1}\int_{\R^d} f(x_1)^{2p-m}dx_1\int_{(\R^d)^{m-1}}\prod_{j=2}^mdy_j\int_{\R^{2p-2m}}\prod_{l=m+1}^pdu_ldv_l\\
&\hspace{3cm}\psi_{\Gamma,1}(0,y_2,\dotsc,y_m,u_{m+1},\dotsc,u_p)\psi_{\Gamma,1}(0,y_2,\dotsc,y_m,v_{m+1},\dotsc,v_p)\\
&\lesssim (t_n^d)^{2p-m-1}\,,
\end{align*}
where we have made use of the translation invariance and scaling property of the kernel $\psi_{\Gamma,t_n}$ as well as of the a.e.-continuity of $f$. The derivation of \eqref{asrelrg2} is similar but easier and is for this reason omitted. First, if $a=b=0$ and $j,m\geq1$, then we have 
\begin{align*}
\norm{g_{\Gamma,t_n}^{(j)}\star_0^0 g_{\Gamma,t_n}^{(m)}}_{L^2(\mu^{\otimes j+m})}^2&=\norm{g_{\Gamma,t_n}^{(j)}}_{L^2(\mu^{\otimes j})}^2\cdot\norm{g_{\Gamma,t_n}^{(m)}}_{L^2(\mu^{\otimes m})}^2\\
&\lesssim (t_n^d)^{2p-j-1}(t_n^d)^{2p-m-1}=(t_n^d)^{4p-(j+m)-2}\,.
\end{align*}
Now note that by the definition of the set $Q(i,k,r,l)$ we further have that 
\[j+m=j+m-a-b\leq i+k-r-l\]
which implies that 
\begin{align*}
\norm{g_{\Gamma,t_n}^{(j)}\star_0^0 g_{\Gamma,t_n}^{(m)}}_{L^2(\mu^{\otimes j+m})}&\lesssim (t_n^d)^{4p-(i+k-r-l)-2}=(t_n^d)^{4p-(i+k+r-l)+2r-2}\\
&\leq (t_n^d)^{4p-(i+k+r-l)}\,,
\end{align*}
since $r\geq1$. If $a=b=j=m=0$, then we have 
\begin{align*}
\norm{g_{\Gamma,t_n}^{(0)}\star_0^0 g_{\Gamma,t_n}^{(0)}}_{L^2(\mu^{\otimes 0})}^2&=\mu^p(\psi_{\Gamma,t_n})^4\lesssim (t_n^d)^{4p-4}\\
&\leq (t_n^d)^{4p-(i+k+r-l)-1}\,,
\end{align*}
which provides a bound of the same order as \eqref{sg1}. If $a=b=j=0$ and $m\geq 1$, then using $m=j+m-a-b\leq i+k-r-l$ and $r\geq1$,
\begin{align*}
\norm{g_{\Gamma,t_n}^{(0)}\star_0^0 g_{\Gamma,t_n}^{(m)}}^2_{L^2(\mu^{\otimes m})}&=\mu^p(\psi_{\Gamma,t_n})^2\norm{g_{\Gamma,t_n}^{(m)}}_{L^2(\mu^{\otimes m})}^2\\
&\lesssim (t_n^d)^{2p-2}(t_n^d)^{2p-m-1}=(t_n^d)^{4p-m-3}\\
&\leq (t_n^d)^{4p-(i+k-r-l)-3}=(t_n^d)^{4p-(i+k+r-l)-3+2r}\\
&\leq (t_n^d)^{4p-(i+k+r-l)-1}\,,
\end{align*}
which again yields a bound of the same order as \eqref{sg1}.
The only remaining possibility is that $1\leq a=b=m=j\leq p-1$. In this case, we first claim that 
\[2j+1\leq i+k+r-l\,.\] 
Indeed, if $j<i$, then $2j<i+k\leq i+k+r-l$ since $j\leq k$. 
On the other hand, if $j=i$, then $j=k$ and we must also have $r=j$ and $l\leq r-1=j-1$ since $j=a\leq r\leq k=j$ and $0\leq l\leq i+k-r-1=j-1=r-1$. Hence, $i+k+r-l\geq 2j+r-l\geq 2j+1$. 
Thus, we obtain that 
\begin{align*}
\norm{g_{\Gamma,t_n}^{(j)}\star_j^j g_{\Gamma,t_n}^{(j)}}^2_{L^2(\mu^{\otimes 0})}&=\norm{g_{\Gamma,t_n}^{(j)}}_{L^2(\mu^{\otimes j})}^4 \\
&\lesssim (t_n^d)^{2p-j-1}(t_n^d)^{2p-j-1}=(t_n^d)^{4p-(2j+1)-1}\\
&\leq (t_n^d)^{4p-(i+k+r-l)-1}
\end{align*}
which is the same bound as in \eqref{sg1}. Since all these bounds are at most of the same order as the bound in \eqref{sg1} we conclude that these estimates indeed hold for all 
$(j,m,a,b)\in Q(i,k,r,l)$. 

Next, we consider the four cases $\textbf{(C1)}$-$\textbf{(C4)}$, separately.
We are going to repeatedly use \eqref{sg1} and Proposition \ref{regprop} for the following estimates:
In case \textbf{(C1)} we have 
\begin{align*}
\frac{n^{4p+l-r-i-k}}{\sigma_n^4}\,\norm{g_{\Gamma,t_n}^{(j)}\star_a^{b}g_{\Gamma,t_n}^{(m)}}_{L^2(\mu^{\otimes j+m-a-b})}^2 
&\lesssim\frac{n^{4p+l-r-i-k} t_n^{d(4p-i-k-r+l-1)}}{n^{2p} t_n^{d(2p-2)}}\\
&=n^{2p-(i+k+r-l)} t_n^{d(2p-(i+k+r-l)+1)}\\
&= (nt_n^d)^{2p-(i+k+r-l)} t_n^d\\
&\lesssim \Bigl(n^pt_n^{d(p-1)}\Bigr)^{-1}\,,
\end{align*}
where we have used that $i+k+r-l\leq 3p$ for the second inequality.
In case \textbf{(C2)} we obtain 
\begin{align*}
\frac{n^{4p+l-r-i-k}}{\sigma_n^4}\,\norm{g_{\Gamma,t_n}^{(j)}\star_a^{b}g_{\Gamma,t_n}^{(m)}}_{L^2(\mu^{\otimes j+m-a-b})}^2 
&\lesssim\frac{n^{4p+l-r-i-k} t_n^{d(4p-i-k-r+l-1)}}{n^{2p} t_n^{d(2p-2)}}\\
&=n^{2p-(i+k+r-l)} t_n^{d(2p-(i+k+r-l)+1)}\\
&\lesssim n^{2p-3} t_n^{d(2p-2)}\,,
\end{align*}
where we have used that $(i+k+r-l)\geq3$. 
In case \textbf{(C3)} we similarly obtain 
\begin{align*}
\frac{n^{4p+l-r-i-k}}{\sigma_n^4}\,\norm{g_{\Gamma,t_n}^{(j)}\star_a^{b}g_{\Gamma,t_n}^{(m)}}_{L^2(\mu^{\otimes j+m-a-b})}^2 
&\lesssim\frac{n^{4p+l-r-i-k} t_n^{d(4p-i-k-r+l-1)}}{n^{4p-2} t_n^{d(4p-4)}}\\
&= n^{2-(i+k+r-l)} t_n^{d(3-(i+k+r-l))}\\
&=n^{-1} (nt_n^d)^{3-(i+k+r-l)}\\
& \lesssim n^{-1}\,,
\end{align*}
where we have again used that $(i+k+r-l)\geq3$. Finally, in case \textbf{(C4)} we have 
\begin{align*}
\frac{n^{4p+l-r-i-k}}{\sigma_n^4}\,\norm{g_{\Gamma,t_n}^{(j)}\star_a^{b}g_{\Gamma,t_n}^{(m)}}_{L^2(\mu^{\otimes j+m-a-b})}^2 
&\lesssim\frac{n^{4p+l-r-i-k} t_n^{d(4p-i-k-r+l-1)}}{n^2}\\
&\leq n^{-1}\bigl(nt_n^d)^{4p+l-r-i-k-1}\\
& \sim n^{-1}\rho^{4p+l-r-i-k-1}\\
&=O(n^{-1})\,.
\end{align*}

 Since the estimates just proven are independent of the variables $k,i,l$ and $r$ this implies that the quantities $B_1'(i,k,n)$ are of the claimed order. Note further that we have the bound $\|\psi^{(n,i)}\|_{L^2(\mu^{\otimes p})}\leq1$ for $1\leq i\leq p$.  Finally, we observe that, in cases \textbf{(C1)} and \textbf{(C2)}, the respective relations $n^p t_n^{d(p-1)}=O(n)$ and $n^{-1}=o\bigl(n^{2p-3}t_n^{d(2p-2)}\bigr)$ 
hold such that the last term in the bound of Theorem \ref{genutheo} does not affect the rate of convergence in these cases. 
\end{proof}

\section{Proofs} \label{proofs}
In this section we outline the proofs of several auxiliary results in the paper.

\begin{proof}[\bf Proof of Lemma \ref{contrlemma}]
For ease of notation, in this proof we will write $\xbf$ and $\wbf$ for elements $\xbf=(x_1,\dotsc,x_l)\in E^l$ and $\wbf=(w_1,\dotsc,w_l)\in E^l$, respectively, $\ybf$ for an element $\ybf=(y_1,\dotsc,y_{r-l})\in E^{r-l}$, 
$\tbf$ for an element $\tbf=(t_1,\dotsc,t_{p-r})\in E^{p-r}$ and $\sbf$ for an element $\sbf=(s_1,\dotsc,s_{q-r})\in E^{q-r}$.
We first prove the results which hold for arbitrary $\sigma$-finite $\mu$. 
By using the Cauchy-Schwarz inequality we have 
\begin{align}\label{cl2}
(\abs{\psi}\star_r^l \abs{\phi})(\ybf,\tbf,\sbf)&=\int_{E^l}\babs{\psi(\xbf, \ybf,\tbf)\cdot\phi(\xbf, \ybf,\sbf)}d\mu^{\otimes l}(\xbf)\notag\\
&\leq \biggl(\int_{E^l}\psi^2(\xbf, \ybf,\tbf)d\mu^{\otimes l}(\xbf)\biggr)^{1/2}\cdot\biggl(\int_{E^l}\phi^2(\xbf, \ybf,\sbf)d\mu^{\otimes l}(\xbf)\biggr)^{1/2}\,.
\end{align}
Now, as $\psi\in L^2(\mu^{\otimes p})$ and $\phi\in L^2(\mu^{\otimes q})$, by the Fubini-Tonelli theorem, the expressions 
\begin{align*}
 &\int_{E^l}\psi^2(\xbf, \ybf,\tbf)d\mu^{\otimes l}(\xbf)\quad\text{and}\\
 &\int_{E^l}\phi^2(\xbf, \ybf,\sbf)d\mu^{\otimes l}(\xbf)
\end{align*}
are finite for $\mu^{\otimes p-l}$-a.a. $(\ybf,\tbf)\in E^{p-l}$ and $\mu^{\otimes q-l}$-a.a. $(\ybf,\sbf)\in E^{q-l}$, respectively. Hence, 
from \eqref{cl2} and as
\begin{align}\label{cl12}
 \babs{\bigl(\psi\star_r^l \phi\bigr)(\ybf,\tbf,\sbf)}\leq \int_{E^l}\babs{\psi(\xbf, \ybf,\tbf)\cdot\phi(\xbf, \ybf,\sbf)}d\mu^{\otimes l}(\xbf)=(\abs{\psi}\star_r^l \abs{\phi})(\ybf,\tbf,\sbf)
\end{align}
we conclude that $\psi\star_r^l \phi$ is well-defined $\mu^{\otimes p+q-r-l}$-a.e. on $E^{p+q-r-l}$ such that (i) is proved. 
Next, by \eqref{cl2} 
and Fubini's theorem, due to nonnegativity, we have 
\begin{align*}
 g(\ybf,\tbf,\sbf)&:=(\abs{\psi}\star_r^l \abs{\phi})^2(\ybf,\tbf,\sbf)\notag\\
&\leq \int_{E^l}\psi^2(\xbf,\ybf,\tbf)d\mu^{\otimes l}(\xbf) \cdot \int_{E^l}\phi^2(\xbf,\ybf,\sbf)d\mu^{\otimes l}(\xbf)\notag\\
&=\int_{E^{l+l}}\psi^2(\xbf,\ybf,\tbf)\phi^2(\wbf,\ybf,\sbf)d\mu^{\otimes l+l}(\xbf,\wbf)   \,.
\end{align*}
Hence, taking into account \eqref{cl12} and, again using Fubini's theorem as well as the Cauchy-Schwarz inequality, we have 
\begin{align}\label{cl13}
 &\int_{E^{p+q-r-l}}\bigl(\psi\star_r^l \phi\bigr)^2(\ybf,\tbf,\sbf)d\mu^{\otimes p+q-r-l}(\ybf,\tbf,\sbf)\notag\\
 &\leq\int_{E^{p+q-r-l}}g(\ybf,\tbf,\sbf)d\mu^{\otimes p+q-r-l}(\ybf,\tbf,\sbf)\notag\\
 &\leq \int_{E^{p+q-r-l}}\biggl(\int_{E^{l+l}}\psi^2(\xbf,\ybf,\tbf)\phi^2(\wbf,\ybf,\sbf)d\mu^{\otimes l+l}(\xbf,\wbf)\biggr)d\mu^{\otimes p+q-r-l}(\ybf,\tbf,\sbf)\notag\\
 &=\int_{E^{r-l}}\biggl(\int_{E^{l+p-r}}\psi^2(\xbf,\ybf,\tbf)d\mu^{\otimes l+p-r}(\xbf,\tbf) \notag\\
 &\hspace{3cm}\cdot \int_{E^{l+q-r}}\phi^2(\wbf,\ybf,\sbf)d\mu^{\otimes l+q-r}(\wbf,\sbf)\biggr)d\mu^{\otimes r-l}(\ybf)\notag\\
 &=\int_{E^{r-l}}\bigl(\psi\star_p^{l+p-r} \psi\bigr)(\ybf)\bigl(\phi\star_q^{l+q-r} \phi\bigr)(\ybf)d\mu^{\otimes r-l}(\ybf)\notag\\
 &\leq \norm{\psi\star_p^{l+p-r} \psi}_{L^2(\mu^{\otimes r-l})}\; \norm{\phi\star_q^{l+q-r}\phi}_{L^2(\mu^{\otimes r-l})}\,,
\end{align}
proving (ii). Item (v) is the special case of (ii) when $r=l$ since $\psi\star_p^{p} \psi=\norm{\psi}_{L^2(\mu^{\otimes p})}^2$ and $\phi\star_q^{q} \phi=\norm{\phi}_{L^2(\mu^{\otimes q})}^2$.  

To prove (vi), first observe that under the assumptions { in the statement (and \eqref{cl13} in particular )}, one has that $g(\ybf,\tbf,\sbf)$ is finite for $\mu^{\otimes p+q-r-l}$-a.a. $(\ybf,\tbf,\sbf)\in E^{p+q-r-l}$. Hence, as
\begin{equation*}
 g(\ybf,\tbf,\sbf)=\int_{E^{l+l}}\babs{\psi(\xbf,\ybf,\tbf)\psi(\wbf,\ybf,\tbf)\phi(\xbf,\ybf,\sbf)\phi(\wbf,\ybf,\sbf)}d\mu^{\otimes l+l}(\xbf,\wbf)\,,
\end{equation*}
\textit{a fortiori}
\begin{align*}
 f(\ybf,\tbf,\sbf)&:=\int_{E^{l+l}}\psi(\xbf,\ybf,\tbf)\psi(\wbf,\ybf,\tbf)\phi(\xbf,\ybf,\sbf)\phi(\wbf,\ybf,\sbf)d\mu^{\otimes l+l}(\xbf,\wbf)
\end{align*}
is well-defined for $\mu^{\otimes p+q-r-l}$-a.a. $(\ybf,\tbf,\sbf)\in E^{p+q-r-l}$. Furthermore, from \eqref{cl13} and the Fubini-Tonelli theorem we conclude that 
\begin{align}\label{cl6}
 (\psi\star_r^l \phi)^2(\ybf,\tbf,\sbf)&=\biggl(\int_{E^l}\psi(\xbf,\ybf,\tbf)\phi(\xbf,\ybf,\sbf) d\mu^{\otimes l}(\xbf)\biggr)^2\notag\\
 &=\int_{E^{l+l}}\psi(\xbf,\ybf,\tbf)\psi(\wbf,\ybf,\tbf)\phi(\xbf,\ybf,\sbf)\phi(\wbf,\ybf,\sbf)d\mu^{\otimes l+l}(\xbf,\wbf)\notag\\
 &=f(\ybf,\tbf,\sbf)
 \end{align}
for $\mu^{\otimes p+q-r-l}$-a.a. $(\ybf,\tbf,\sbf)\in E^{p+q-r-l}$. Note that from \eqref{cl13} we also have that $\psi\star_r^l \phi\in L^2(\mu^{\otimes p+q-r-l})$ because the right-hand side of the inequality is finite by assumption. 
Moreover, \eqref{cl13}, \eqref{cl6} and the Fubini-Tonelli theorem assure that we can interchange the order of integration in the following computation:
\begin{align}\label{cl8}
 &\int_{E^{p+q-r-l}}(\psi\star_r^l \phi)^2(\ybf,\tbf,\sbf)d\mu^{\otimes p+q-r-l}(\ybf,\tbf,\sbf)\notag\\
 &=\int_{E^{p+q-r-l}}f(\ybf,\tbf,\sbf)d\mu^{\otimes p+q-r-l}(\ybf,\tbf,\sbf)\notag\\
 &=\int_{E^{l+l}}\biggl(\int_{E^{r-l}}\Bigl(\int_{E^{p+q-2r}}\psi(\xbf,\ybf,\tbf)\psi(\wbf,\ybf,\tbf)\phi(\xbf,\ybf,\sbf)\phi(\wbf,\ybf,\sbf)d\mu^{\otimes p+q-2r}(\tbf,\sbf)\Bigr)\notag\\
 &\hspace{5cm} d\mu^{\otimes r-l}(\ybf)\biggr)d\mu^{\otimes l+l}(\xbf,\wbf)\notag\\
 &=\int_{E^{l+l}}\biggl(\int_{E^{r-l}}\Bigl(\int_{E^{p-r}}\psi(\xbf,\ybf,\tbf)\psi(\wbf,\ybf,\tbf)d\mu^{\otimes p-r}(\tbf)\Bigr) \notag\\
 &\hspace{3cm}\Bigl(\int_{E^{q-r}}\phi(\xbf,\ybf,\sbf)\phi(\wbf,\ybf,\sbf)d\mu^{\otimes q-r}(\sbf)\Bigr) d\mu^{\otimes r-l}(\ybf)\biggr)d\mu^{\otimes l+l}(\xbf,\wbf)\notag\\
 &=\int_{E^{l+l}}\biggl(\int_{E^{r-l}}(\psi\star_{p-l}^{p-r}\psi)(\xbf,\wbf,\ybf) (\phi\star_{q-l}^{q-r}\phi)(\xbf,\wbf,\ybf) d\mu^{\otimes r-l}(\ybf)\biggr)d\mu^{\otimes l+l}(\xbf,\wbf)\notag\\
 &=\int_{E^{r+l}}(\psi\star_{p-l}^{p-r}\psi)(\phi\star_{q-l}^{q-r}\phi)d\mu^{\otimes r+l}\,,
\end{align}
which proves the equality in (vi). 
Next, we apply the Cauchy-Schwarz inequality to \eqref{cl8} and obtain
\begin{align}\label{cl9}
 \|\psi\star_r^l\phi\|_{L^2(\mu^{\otimes p+q-r-l})}^2& \leq\|\psi\star_{p-l}^{p-r}\psi\|_{L^2(\mu^{\otimes r+l})}\cdot  \|\phi\star_{q-l}^{q-r}\phi\|_{L^2(\mu^{\otimes r+l})}\,.
\end{align}
Now, repeating the same arguments and computations with $\psi=\phi$ we can conclude the remaining parts of statement (vi), noting in particular that, due to \eqref{cl8}, we have
\begin{align}
 \|\psi\star_{p-l}^{p-r}\psi\|_{L^2(\mu^{\otimes r+l})}&=\|\psi\star_{r}^{l}\psi\|_{L^2(\mu^{\otimes 2p-r-l})}\quad\text{and}\label{cl10}\\
 \|\phi\star_{p-l}^{p-r}\phi\|_{L^2(\mu^{\otimes r+l})}&=\|\phi\star_{r}^{l}\phi\|_{L^2(\mu^{\otimes 2q-r-l})}\label{cl11}\,.
\end{align}
The inequality in the statement (vi) now follows from \eqref{cl9}, \eqref{cl10} and \eqref{cl11}. 

Next, we prove (iii) and (iv) which hold for probability measures $\mu$. In order to do this, we first apply Jensen's inequality to \eqref{defcontr2} which gives
\begin{align}
&\bigl(\psi\star_r^l \phi\bigr)^2(\ybf,\tbf,\sbf)\notag\\
&\leq \E\Bigl[\psi^2\bigl(X_1,\dotsc,X_l, \ybf,\tbf\bigr)\cdot\phi^2\bigl(X_1,\dotsc,X_l, \ybf,\sbf\bigr)\Bigr]\label{cl3}\\
&=\int_{E^l}\Bigl(\psi^2(\xbf,\ybf,\tbf)\cdot\phi^2(\xbf,\ybf,\sbf)\Bigr)d\mu^{\otimes l}(\xbf)\label{cl1}\,.
\end{align}
To prove (iii), we use \eqref{cl3} and the Fubini-Tonelli theorem as well as the Cauchy-Schwarz inequality to obtain
\begin{align*}
& \|\psi\star_r^l\phi\|_{L^2(\mu^{\otimes p+q-r-l})}^2\notag\\
&\leq \int_{E^{p+q-r-l}}\E\Bigl[\psi^2\bigl(X_1,\dotsc,X_l, \ybf,\tbf\bigr)\phi^2\bigl(X_1,\dotsc,X_l, \ybf,\sbf\bigr)\Bigr]\\
&\hspace{4cm}d\mu^{\otimes p+q-r-l}(\ybf,\tbf,\sbf)\\
&=\E\Bigl[\psi^2\bigl(X_1,\dotsc,X_l, Y_1,\dotsc,Y_{r-l},T_1,\dotsc,T_{p-r}\bigr)\\
&\hspace{2cm}\cdot\phi^2\bigl(X_1,\dotsc,X_l, Y_1,\dotsc,Y_{r-l},S_1,\dotsc,S_{q-r}\bigr)\Bigr]\\
&=\E\Bigl[\bigl(\psi\star_p^{p-r}\psi\bigr) \bigl(X_1,\dotsc,X_r\bigr)\cdot \bigl(\phi\star_q^{q-r}\phi\bigr) \bigl(X_1,\dotsc,X_r\bigr)\Bigr]\\
&\leq\|\psi\star_p^{p-r}\psi\|_{L^2(\mu^{\otimes r})} \cdot\|\phi\star_q^{q-r}\phi\|_{L^2(\mu^{\otimes r})} \,,
\end{align*}
where we let $Y_j:=X_{l+j}$, $j=1,\dotsc,r-l$, $T_j:=X_{r+j}$, $j=1,\dotsc,p-r$, and $S_j:=X_{p+j}$, $j=1,\dotsc,q-r$.

Now, in order to prove (iv) note that by \eqref{cl1} we have 
\begin{align}\label{cl4}
& \int_{E^{p+q-r-l}}\bigl(\psi\star_r^l \phi\bigr)^2d\mu^{\otimes p+q-r-l}\notag\\
 &\leq \int_{E^{p+q-r-l}}\biggl(\int_{E^l}\Bigl(\psi^2(\xbf, \ybf,\tbf)\cdot\phi^2(\xbf,\ybf,\sbf)\Bigr)d\mu^{\otimes l}(\xbf)\biggr)\notag\\
 &\hspace{4cm}d\mu^{\otimes p+q-r-l}(\ybf,\tbf,\sbf)\,.
\end{align}
Hence, by applying the Cauchy-Schwarz inequality on \eqref{cl4} and then again to obtain the second inequality we have 
\begin{align*}
&\int_{E^{p+q-r-l}}\bigl(\psi\star_r^l \phi\bigr)^2d\mu^{\otimes p+q-r-l}\\
&\leq \int_{E^{p+q-r-l}}\biggl(\int_{E^l}\psi^4(\xbf,\ybf,\tbf)d\mu^{\otimes l}(\xbf)\biggr)^{1/2}\\
&\hspace{3cm}\biggl(\int_{E^l}\phi^4(\xbf,\ybf,\sbf)d\mu^{\otimes l}(\xbf)\biggr)^{1/2}d\mu^{\otimes p+q-r-l}(\ybf,\tbf,\sbf)\\
&\leq \biggl(\int_{E^{p+q-r}}\psi^4(\xbf,\ybf,\tbf)d\mu^{\otimes p+q-r}(\xbf,\ybf,\tbf,\sbf)\biggr)^{1/2}\\
&\hspace{3cm}\biggl(\int_{E^{p+q-r}}\phi^4(\xbf,\ybf,\tbf)d\mu^{\otimes p+q-r}(\xbf,\ybf,\tbf,\sbf)\biggr)^{1/2}\\
&=\mu(E)^{(q-r)/2}\Bigl(\int_{E^p}\psi^4d\mu^{\otimes p}\Bigr)^{1/2}\cdot \mu(E)^{(p-r)/2}\Bigl(\int_{E^q}\phi^4d\mu^{\otimes q}\Bigr)^{1/2}\notag\\
&=\Bigl(\int_{E^p}\psi^4d\mu^{\otimes p}\Bigr)^{1/2}\cdot \Bigl(\int_{E^q}\phi^4d\mu^{\otimes q}\Bigr)^{1/2}\notag\\
&=\|\psi\|_{L^4(\mu^{\otimes p})}^2\|\phi\|_{L^4(\mu^{\otimes q})}^2\,,
\end{align*}
proving (iv).
\end{proof}

\begin{proof}[\bf Proof of Proposition \ref{pform}]
Write
\begin{equation*}
W:=J_p(\psi)=\sum_{J\in\D_p}W_J\quad\text{and}\quad V:=J_q(\phi)=\sum_{K\in\D_q} V_K
\end{equation*}
for the respective Hoeffding decompositions of $W$ and $V$. From Theorem 2.6 in \cite{DP16} we know that the Hoeffding decomposition of $VW$ is given by 
\begin{equation*}
 VW=\sum_{\substack{M\subseteq[n]:\\\abs{M}\leq p+q}} U_M\,,
\end{equation*}
where, for $M\subseteq[n]$ with $\abs{M}\leq p+q$ we have 
\begin{align}\label{hdvw}
U_M&=\sum_{\substack{J\in\D_p,K\in\D_q:\\J\Delta K\subseteq M\subseteq J\cup K}}
\sum_{\substack{L\subseteq[n]:\\ J\Delta K\subseteq L\subseteq M}}(-1)^{\abs{M}-\abs{L}}\E\bigl[W_JV_K\,\bigl|\,\F_L\bigr]\notag\\
&=\sum_{L\subseteq M}(-1)^{\abs{M}-\abs{L}}
\sum_{\substack{J\in\D_p,K\in\D_q:\\J\Delta K\subseteq L,\\
M\subseteq J\cup K}}\E\bigl[W_JV_K\,\bigl|\,\F_L\bigr]\,.
\end{align}
Note that $U_M=0$ a.s. whenever $\abs{M}<\abs{p-q}$ because $\abs{J\Delta K}\geq \abs{p-q}$ for all $J\in\D_p$ and $K\in\D_q$. Let us fix for the moment sets $L, M\subseteq[n]$ as well as $J\in\D_p$ and $K\in\D_q$ such that 
\[J\Delta K\subseteq L\subseteq M\subseteq J\cup K\,.\]
Write 
\begin{equation*}
 m:=\abs{M}\,,\quad s:=\abs{L}\quad\text{and}\quad r:=\abs{J\cap K}\leq p\wedge q\,.
\end{equation*}

Note that we have $\abs{J\cup K}=\abs{J}+\abs{K}-\abs{J\cap K}=p+q-r$ and  
\begin{equation*}
 \abs{p-q}\leq s\leq m\leq p+q-r\,.
\end{equation*}
Furthermore, we have 
\begin{align*}
 \abs{J\Delta K}&=\abs{J\cup K}-\abs{J\cap K}=p+q-2r
\end{align*}
and
\begin{align*}
 l&:=\babs{(J\cap K)\setminus L}=\babs{(J\cup K)\setminus L}-\babs{(J\Delta K)\setminus L}=\babs{(J\cup K)\setminus L}\\
 &=\abs{J\cup K}-\abs{L}=p+q-r-s\,.
\end{align*}
Also,
\begin{align}\label{hdvw2}
 \E\bigl[W_JV_K\,\bigl|\,\F_L\bigr]&=\E\bigl[\psi(X_j,j\in J)\phi(X_k,k\in K)\,\bigl|\,X_i,i\in L\bigr]\notag\\
 &=\psi\star_r^l \phi\bigl((X_i)_{i\in L\cap J\cap K},(X_j)_{j\in J\setminus K}, (X_k)_{k\in K\setminus J}\bigr)\,.
\end{align}

We denote  by $\Pi_r(L)$ the collection of all (ordered) partitions $(A,B,C)$ of the set $L$ (i.e. $L$ is the disjoint union of $A$, $B$ and $C$) such that 
$\abs{A}=2r+s-p-q$, $\abs{B}=p-r$ and $\abs{C}=q-r$. Then, for given sets $L\subseteq M\subseteq [n]$ with $\abs{L}=s\leq m=\abs{M}$, a fixed $r\in\{0,1,\dotsc,p\wedge q\}$ and for a 
fixed triple $(A,B,C)\in\Pi_r(L)$ there are exactly 
\begin{equation*}
 \binom{n-\abs{M}}{\abs{J\cup K}-\abs{M}}=\binom{n-m}{p+q-r-m}
\end{equation*}
pairs $(J,K)\in\D_p\times\D_q$ such that $\abs{J\cap K}=r$, $J\cap K\cap L=A$, $J\setminus K=B$, $K\setminus J=C$ and $M\subseteq J\cup K$. Indeed, given these restrictions it only remains to choose the 
set $(J\cap K)\setminus L$ such that 
\[M\setminus L\subseteq (J\cap K)\setminus L\,.\]
The claim now follows from the fact that 
\begin{equation*}
 \babs{(J\cap K)\setminus L}-\babs{M\setminus L}=p+q-r-s-(m-s)=p+q-r-m\,.
\end{equation*}
The above implies that, still for fixed $L$ and $M$, we have
\begin{align}\label{pf1}
& \sum_{\substack{J\in\D_p,K\in\D_q:\\J\Delta K\subseteq L,\\
M\subseteq J\cup K,\\\abs{J\cap K}=r}}\psi\star_r^{l}\phi\bigl((X_i)_{i\in L\cap J\cap K},(X_j)_{j\in J\setminus K}, (X_k)_{k\in K\setminus J}\bigr)\notag\\
&=\binom{n-m}{p+q-r-m}\sum_{(A,B,C)\in\Pi_r(L)}\psi\star_r^{l}\phi\bigl((X_i)_{i\in A}, (X_i)_{i\in B}, (X_i)_{i\in C}\bigr)\,.
\end{align}

Let us assume that $M=\{j_1,\dotsc,j_m\}$ with $1\leq j_1<\ldots<j_m\leq n$. For $\pi\in\Sm$ let us write 
\begin{equation*}
 T_\pi:=\E\Bigl[\bigl(\psi\star_r^{p+q-r-m}\phi\bigr)\bigl(X_{j_{\pi(1)}},\dotsc,X_{j_{\pi(m)}}\bigr)\,\Bigl|\,(X_j)_{j\in L}\Bigr]\,.
\end{equation*}
Then we have that $T_\pi=0$ unless
\[\{j_{\pi(2r+m-p-q+1)},\dotsc,j_{\pi(m)}\}\subseteq L\,.\]
Also, for a given partition $(A,B,C)\in\Pi_r(L)$ there are $(p-r)!(q-r)!(m-p-q+2r)!$ permutations $\pi\in\mathbb{S}_m$ such that
\begin{equation*}
 T_\pi=\bigl(\psi\star_r^{l}\phi\bigr)\bigl((X_i)_{i\in A}, (X_i)_{i\in B}, (X_i)_{i\in C}\bigr)\,.
\end{equation*}
(More precisely, there are exactly $(p-r)!(q-r)!(m-p-q+2r)!$ permutations $\pi\in\mathbb{S}_m$ such that 
\begin{align*}
 \{j_{\pi(1)},\dotsc,j_{\pi(2r+m-p-q)}\}\cap L&=A,\\
 \{j_{\pi(2r+m-p-q+1)},\dotsc,j_{\pi(r+m-q)}\}&=B,\quad\text{and}\\
 \{j_{\pi(r+m-q+1)},\dotsc,j_{\pi(m)}\}&=C.\text{)}
\end{align*}

Hence, we conclude that 
\begin{align}\label{pf2}
 &T_{M,L}(r):=\E\Bigl[\bigl(\widetilde{\psi\star_r^{p+q-r-m}\phi}\bigr)\bigl(X_j,j\in M\bigr)\,\Bigl|\,(X_j)_{j\in L}\Bigr]=\frac{1}{m!}\sum_{\pi\in\mathbb{S}_m} T_\pi\notag\\
 &=\frac{(p-r)!(q-r)!(m-p-q+2r)!}{m!}\notag\\
 &\hspace{3cm}\cdot\sum_{(A,B,C)\in\Pi_r(L)}\bigl(\psi\star_r^{l}\phi\bigr)\bigl((X_i)_{i\in A}, (X_i)_{i\in B}, (X_i)_{i\in C}\bigr)\notag\\
 &=\binom{m}{p-r,q-r,m-p-q+2r}^{-1}\notag\\
 &\hspace{3cm}\sum_{(A,B,C)\in\Pi_r(L)}\bigl(\psi\star_r^{l}\phi\bigr)\bigl((X_i)_{i\in A}, (X_i)_{i\in B}, (X_i)_{i\in C}\bigr)\,.
 \end{align}
From \eqref{pf1} and \eqref{pf2} we thus { deduce that}, for fixed sets $L\subseteq M\subseteq [n]$ with $\abs{p-q}\leq\abs{L}=s\leq m=\abs{M}$ and for a fixed $r\in\{0,1,\dotsc,p\wedge q\}$,
\begin{align}\label{r2e1}
 & \sum_{\substack{J\in\D_p,K\in\D_q:\\J\Delta K\subseteq L,\\
M\subseteq J\cup K,\\\abs{J\cap K}=r}}\psi\star_r^{l}\phi\bigl((X_i)_{i\in L\cap J\cap K},(X_j)_{j\in J\setminus K}, (X_k)_{k\in K\setminus J}\bigr)\notag\\
&=\binom{n-m}{p+q-r-m}\binom{m}{p-r,q-r,m-p-q+2r}T_{M,L}(r)\,.
\end{align}

Hence, \eqref{hdvw}, \eqref{hdvw2} and \eqref{r2e1} together imply that for $M\subseteq[n]$ with\\
$\abs{p-q}\leq m:=\abs{M}\leq p+q$ we have
\begin{align*}
 U_M&=\sum_{s=\abs{p-q}}^m \sum_{\substack{L\subseteq M:\\ \abs{L}=s}}(-1)^{m-s}\sum_{r=\ceil{\frac{p+q-s}{2}}}^{p\wedge q\wedge(p+q-m)}\binom{n-m}{p+q-r-m}\notag\\
 &\hspace{3cm}\cdot\binom{m}{p-r,q-r,m-p-q+2r}T_{M,L}(r)\notag\\
 &=\sum_{r=\ceil{\frac{p+q-m}{2}}}^{p\wedge q\wedge(p+q-m)}\binom{n-m}{p+q-r-m}\binom{m}{p-r,q-r,m-p-q+2r}\notag\\
 &\hspace{3cm}\cdot\sum_{s=p+q-2r}^m(-1)^{m-s}\sum_{\substack{L\subseteq M:\\ \abs{L}=s}} T_{M,L}(r)\notag\\
 &=\sum_{r=\ceil{\frac{p+q-m}{2}}}^{p\wedge q\wedge(p+q-m)}\binom{n-m}{p+q-r-m}\binom{m}{p-r,q-r,m-p-q+2r}\notag\\
 &\hspace{3cm}\cdot\Bigl(\widetilde{\psi\star_r^{p+q-r-m}\phi}\Bigr)_m\bigl(X_j,j\in M\bigr)\,,
\end{align*}
where we have used that 
\begin{align*}
 \sum_{s=p+q-2r}^m(-1)^{m-s}\sum_{\substack{L\subseteq M:\\ \abs{L}=s}} T_{M,L}(r)&=\sum_{L\subseteq M}(-1)^{m-\abs{L}}T_{M,L}(r)\\
 &=\Bigl(\widetilde{\psi\star_r^{p+q-r-m}\phi}\Bigr)_m\bigl(X_j,j\in M\bigr)
\end{align*}
for the last identity. Thus, we obtain that 
\begin{align*}
 \sum_{\substack{M\subseteq[n]:\\\abs{M}=m}}U_M&=\sum_{r=\ceil{\frac{p+q-m}{2}}}^{p\wedge q\wedge(p+q-m)}\binom{n-m}{p+q-r-m}\binom{m}{p-r,q-r,m-p-q+2r}\notag\\
 &\hspace{3cm}\cdot\sum_{\substack{M\subseteq[n]:\\ \abs{M}=m}} \Bigl(\widetilde{\psi\star_r^{p+q-r-m}\phi}\Bigr)_m\bigl(X_j,j\in M\bigr)\notag\\
 &=J_m(\chi_m)\,,
 \end{align*}
yielding the claim.
\end{proof}

The next two proofs rely on Stein's method of exchangeable pairs for univariate and multivariate normal approximation, respectively. 
As most parts of the proofs are implicit in the paper \cite{DP16}, we keep the presentation as short as possible. 
Recall that a pair $(X,X')$ of random elements, defined on the same probability space $(\Omega,\F,\Prob)$, is called \textit{exchangeable}, whenever 
\[(X,X')\stackrel{\D}{=}(X',X)\,.\]
Henceforth, denote by $X=(X_1,\dotsc,X_n)$ our given vector of independent random variables $X_1,\dotsc,X_n$ on $(\Omega,\F,\Prob)$ with values in the respective measurable spaces $(E_1,\mathcal{E}_1),\dotsc,(E_n,\mathcal{E}_n)$. 
In fact, in this paper, $X_1,\dotsc,X_n$ are even i.i.d. with values in $(E,\mathcal{E})$ but we prefer keeping the general framework for possible future reference. 

Let $Y:=(Y_1,\dotsc,Y_n)$ be an independent copy of $X$ and let $\alpha$ be uniformly distributed on $[n]=\{1,\dotsc,n\}$ in such a way that $X,Y,\alpha$ are independent random variables. Letting, for $j=1,\dotsc,n$, 
\begin{equation*}
 X_j':=\begin{cases}
        Y_j\,,&\text{if } \alpha=j\\
        X_j\,,&\text{if }\alpha\not=j
       \end{cases}
\end{equation*}
and
\begin{equation*}
 X':=(X_1',\dotsc,X_n')
\end{equation*}
it is easy to see that the pair $(X,X')$ is exchangeable. 
\begin{proof}[{\bf Proof of Lemma \ref{1dimlemma}}]
 We apply the following variant of Theorem 1, Lecture 3 in \cite{St86} (see \cite{DP16} for more information).  

 \begin{theorem}\label{1dimplugin}
 Let $(W,W')$ be an exchangeable pair of square-integrable, real-valued random variables on $(\Omega,\F,\Prob)$ such that, for some $\lambda>0$ and some sub-$\sigma$-field $\mathcal{G}$ of $\F$ with $\sigma(W)\subseteq\mathcal{G}$, 
 the \textit{linear regression property} 
 \begin{equation}\label{linregprop}
  \E\bigl[W'-W\,\bigl|\,\mathcal{G}\bigr]=-\lambda W
 \end{equation}
 is satisfied. Then, we have that 
\begin{align}\label{1dexbound}
 d_\W(W,Z)&\leq\sqrt{\frac{2}{\pi}}\sqrt{\Var\Bigl(\frac{1}{2\lambda}\E\bigl[(W'-W)^2\,\bigl|\,\mathcal{G}\bigr]\Bigr)} +\frac{1}{3\lambda}\E\babs{W'-W}^3\,.
\end{align}
\end{theorem}
With the exchangeable pair $(X,X')$ from above we construct $W'$ by defining 
\begin{equation*}
 W':=J_{p,X'}(\phi)=\sum_{J\in\D_p}\phi(X_j',j\in J)=\sigma_n^{-1}\sum_{J\in\D_p}\psi(X_j',j\in J)\,.
\end{equation*}
As exchangeability is preserved under functions, the pair $(W,W')$ is clearly exchangeable.
In Lemma 2.3 of \cite{DP16}, we showed that 
\begin{equation}\label{linreg}
 \E\bigl[W'-W\,\bigl|\,X\bigr]=-\frac{p}{n}W\,,
\end{equation}
i.e. \eqref{linregprop} holds with $\mathcal{G}=\sigma(X)$ and $\lambda=p/n$. Furthermore, denoting by 
\begin{equation*}
 W^2=\sum_{\substack{M\subseteq[n]:\\ \abs{M}\leq 2p}}U_M
\end{equation*}
the Hoeffding decomposition of $W^2$, Lemma 2.7 of \cite{DP16} gives the following Hoeffding decomposition: 
\begin{equation}\label{hdcond}
 \frac{n}{2p}\E\bigl[(W'-W)^2\,\bigl|\,X\bigr]=\sum_{\substack{M\subseteq[n]:\\ \abs{M}\leq 2p-1}}a_M U_M\,,
\end{equation}
where 
\[a_M=1-\frac{\abs{M}}{2p}\in[0,1]\quad\text{for each}\quad M\subseteq[n]\text{ with }\abs{M}\leq2p\,.\]
Hence, we conclude that 
\begin{equation}\label{1dle1}
 \Var\Bigl(\frac{n}{2p}\E\bigl[(W'-W)^2\,\bigl|\,X\bigr]\Bigr)=\sum_{\substack{M\subseteq[n]:\\ \abs{M}\leq 2p-1}}a_M^2 \Var\bigl(U_M\bigr)\leq
 \sum_{\substack{M\subseteq[n]:\\ \abs{M}\leq 2p-1}} \Var\bigl(U_M\bigr)\,,
\end{equation}
which already bounds the first term appearing on the right hand side of \eqref{1dexbound}. To bound the second term, we first use the Cauchy-Schwarz inequality to obtain that 
\begin{align}\label{1dle2}
 \frac{1}{3\lambda}\E\babs{W'-W}^3&\leq \frac{n}{3p}\Bigl(\E\bigl[(W'-W)^2\bigr]\Bigr)^{1/2}\Bigl(\E\babs{W'-W}^4\Bigr)^{1/2}\notag\\
 &=\frac{2\sqrt{2}}{3}\Bigl(\frac{n}{4p}\E\bigl[(W'-W)^4\bigr]\Bigr)^{1/2}\,.
\end{align}
Lemma 2.2 in \cite{DP16} implies that 
\begin{align*}
  \frac{n}{4p}\E\bigl[(W'-W)^4\bigr]&=3\E\Bigl[W^2\frac{n}{2p}\E\bigl[(W'-W)^2\,\bigl|\,X\bigr]\Bigr]
	-\E\bigl[W^4\bigr]\,.
\end{align*}
Hence, using the orthogonality of the Hoeffding decomposition we obtain that 
\begin{align}\label{1dle3}
  \frac{n}{4p}\E\bigl[(W'-W)^4\bigr]&=3\sum_{\substack{M,N\subseteq[n]:\\\abs{M},\abs{N}\leq 2p}}a_M\E\bigl[U_MU_N\bigr]-\E\bigl[W^4\bigr]\notag\\
&=3a_\emptyset U_\emptyset^2-\E\bigl[W^4\bigr]+3\sum_{\substack{M\subseteq[n]:\\1\leq\abs{M}\leq 2p}}a_M\Var(U_M)\notag\\
&=3-\E\bigl[W^4\bigr]+3\sum_{\substack{M\subseteq[n]:\\1\leq\abs{M}\leq 2p-1}}a_M\Var(U_M)\notag\\
&\leq 3-\E\bigl[W^4\bigr]+\sum_{\substack{M\subseteq[n]:\\1\leq\abs{M}\leq 2p-1}}\Var(U_M)+2\sum_{\substack{M\subseteq[n]:\\1\leq\abs{M}\leq 2p-1}}a_M\Var(U_M)      \,.
\end{align}
From Lemma 2.10 of \cite{DP16} it follows that 
\begin{align*}
 \sum_{\abs{M}\leq 2p-1}\Var(U_M)&\leq \E[W^4]-3+\kappa_p \rho^2\,,
\end{align*}
where $\kappa_p\in (0,\infty)$ is a constant which only depends on $p$ and where, in the present case,
\begin{equation*}
 \rho^2:=\rho_n^2:=\max_{1\leq i\leq n}\sum_{i\in K\in\D_p}\E[W_K^2]=\binom{n-1}{p-1}\E\bigl[\phi^2(X_1,\dotsc,X_p)\bigr]
 =\frac{p}{n}\,.
\end{equation*}
Thus, from \eqref{1dle3} we conclude that 
\begin{align}\label{1dle4}
 \frac{n}{4p}\E\bigl[(W'-W)^4\bigr]&\leq 2\sum_{\substack{M\subseteq[n]:\\1\leq\abs{M}\leq 2p-1}}a_M\Var(U_M)+\frac{p\kappa_p}{n}\notag\\
 &\leq  2\sum_{\substack{M\subseteq[n]:\\1\leq\abs{M}\leq 2p-1}}\Var(U_M)+\frac{p\kappa_p}{n}\,.
\end{align}
The claim now follows from \eqref{1dle1}, \eqref{1dle2} and \eqref{1dle4}\,.
\end{proof}

\begin{proof}[\bf Proof of Lemma \ref{mdimlemma}]
For the proof of the multivariate lemma we quote the following (simplified) result from \cite{Doe12c}. It is { a} variant of Theorem 3 in \cite{Meck09} albeit with better constants.
 \begin{theorem}\label{meckes}
Let $(W,W')$ be an exchangeable pair of $\R^d$-valued $L^2(\Prob)$ random vectors defined on a probability space $(\Omega,\F,\Prob)$ and let $\G\subseteq\F$ be 
a sub-$\sigma$-field of $\F$ such that $\sigma(W)\subseteq\G$. Suppose there exists a non-random invertible matrix $\Lambda\in\R^{d\times d}$ such that
 the linear regression property
\begin{equation}\label{linreggmn}
\E\bigl[W'-W\,\bigl|\, \G\bigr]=-\Lambda W
\end{equation}
 holds and, for a given non-random positive semidefinite matrix $\Sigma$, define the $\G$-measurable random matrix $S$ by 
\begin{equation}\label{condrm1}
 \E\Bigl[(W'-W)(W'-W)^T\,\Bigl|\,\G\Bigr]=2\Lambda\Sigma +S\,.
\end{equation}
 Finally, denote by $Z$ a centered $d$-dimensional Gaussian vector with covariance matrix $\Sigma$.
\begin{enumerate}[{\normalfont (a)}]
 \item For any $h\in C^3(\R^d)$ such that $\E\bigl[\abs{h(W)}\bigr]<\infty$ and $\E\bigl[\abs{h(Z)}\bigr]<\infty$, 
\begin{align*}
&\bigl|\E[h(W)]-\E[h(Z)]\bigr|\leq\Opnorm{\Lambda^{-1}}\Biggl(\frac{1}{4}\tilde{M}_2(h) \E\bigl[\HSnorm{S}\bigr]
+\frac{1}{18}M_3(h)\E\bigl[\Enorm{W'-W}^3\bigr]\Biggr)\\
&\leq\Opnorm{\Lambda^{-1}}\Biggl(+\frac{\sqrt{d}}{4}M_2(h) \E\bigl[\HSnorm{S}\bigr]+\frac{1}{18}M_3(h)\E\bigl[\Enorm{W'-W}^3\bigr]\Biggr)\,.
\end{align*}
\item If $\Sigma$ is actually positive definite, then for each $h\in C^2(\R^d)$ such that\\ $\E\bigl[\abs{h(W)}\bigr]<\infty$ and $\E\bigl[\abs{h(Z)}\bigr]<\infty$ we have
\begin{align*}
 \bigl|\E[h(W)]-\E[h(Z)]\bigr| 
&\leq M_1(h)\Opnorm{\Lambda^{-1}}\Biggl(\frac{\Opnorm{\Sigma^{-1/2}}}{\sqrt{2\pi}} \E\bigl[\HSnorm{S}\bigr]\Biggr)\\
&\quad+\frac{\sqrt{2\pi}}{24} M_2(h)\Opnorm{\Lambda^{-1}}\Opnorm{\Sigma^{-1/2}}\E\bigl[\Enorm{W'-W}^3\bigr]\,.
\end{align*}
\end{enumerate}
\end{theorem}

\medskip

We now apply Theorem \ref{meckes} with the $\sigma$-field $\mathcal{G}=\sigma(X)$ and the nonnegative definite matrix $\Sigma:=\mathbb{V}=\Cov(W)$. Similarly to the above, we define the random vector 
\[W':=\bigl(W'(1),\dotsc,W'(d)\bigr)^T\]
via 
\begin{equation*}
 W'(i):=J_{p_i,X'}(\psi_i)=\sum_{J\in\D_{p_i}}\psi_i(X'_j,j\in J)\,,\quad 1\leq i\leq d\,.
\end{equation*}
Then, clearly the pair $(W,W')$ is exchangeable and from Lemma 3.2 in \cite{DP16} we know that 
\begin{equation*}
 \E\bigl[W'-W\,\bigl|\,X\bigr]=-\Lambda W
\end{equation*}
holds with 
\begin{equation*}
 \Lambda=\diag\Bigl(\frac{p_1}{n},\dotsc,\frac{p_d}{n}\Bigr)
\end{equation*}
and that the matrix
\begin{equation*}
 S= \E\Bigl[(W'-W)(W'-W)^T\,\Bigl|\,\G\Bigr]-2\Lambda\mathbb{V}
\end{equation*}
is centered. Note that we have
\begin{equation*}
 \Opnorm{\Lambda^{-1}}=\frac{n}{p_1}\,.
\end{equation*}
We start by bounding 
\begin{equation*}
 \Opnorm{\Lambda^{-1}}\E\bigl[\HSnorm{S}\bigr]\,.
\end{equation*}
Note that, since $S$ is centered, using the Cauchy-Schwarz inequality, we obtain
\begin{align}\label{S1}
 \E\bigl[\HSnorm{S}\bigr]&=\E\Bigl[\Bigl(\sum_{i,k=1}^d S_{i,k}^2\Bigr)^{1/2}\Bigr]\leq\Bigl(\sum_{i,k=1}^d\E\bigl[S_{i,k}^2\bigr]\Bigr)^{1/2}=\Bigl(\sum_{i,k=1}^d\Var\bigl(S_{i,k}\bigr)\Bigr)^{1/2}\notag\\
 &=\biggl(\sum_{i,k=1}^d\Var\Bigl(\E\bigl[\bigl(W'(i)-W(i)\bigr)\bigl(W'(k)-W(k)\bigr)\,\bigl|\,X\bigr]\Bigr)\biggr)^{1/2}\,.
\end{align}
By Lemma 3.3 of \cite{DP16} the Hoeffding decomposition of 
\[n\E\bigl[\bigl(W'(i)-W(i)\bigr)\bigl(W'(k)-W(k)\bigr)\,\bigl|\,X\bigr]\]
is given by 
\begin{equation*}
 n\E\bigl[\bigl(W'(i)-W(i)\bigr)\bigl(W'(k)-W(k)\bigr)\,\bigl|\,X\bigr]=\sum_{\substack{M\subseteq[n]:\\\abs{M}\leq p_i+p_k-1}}\bigl(p_i+p_k-\abs{M}\bigr)U_M(i,k)\,,
\end{equation*}
where we recall that 
\begin{equation*}
 W(i)W(k)=\sum_{\substack{M\subseteq[n]:\\\abs{M}\leq p_i+p_k}}U_M(i,k)
\end{equation*}
is the Hoeffding decomposition of $W(i)W(k)$. Hence, by the orthogonality of the terms in the Hoeffding decomposition we obtain, for $1\leq i,k\leq d$,
\begin{align*}
 &\Var\Bigl(n\E\bigl[\bigl(W'(i)-W(i)\bigr)\bigl(W'(k)-W(k)\bigr)\,\bigl|\,X\bigr]\Bigr)\\
 &=\sum_{\substack{M\subseteq[n]:\\\abs{M}\leq p_i+p_k-1}}\bigl(p_i+p_k-\abs{M}\bigr)^2\Var\bigl(U_M(i,k)\bigr)\\
 &=\sum_{q=0}^{p_i+p_k-1}\bigl(p_i+p_k-q\bigr)^2\sum_{\substack{M\subseteq[n]:\\\abs{M}=q}}\Var\bigl(U_M(i,k)\bigr)\\
 &\leq\bigl(p_i+p_k\bigr)^2\sum_{\substack{M\subseteq[n]:\\\abs{M}\leq p_i+p_k-1}}\Var\bigl(U_M(i,k)\bigr)\,.
\end{align*}
From \eqref{S1} we thus conclude that 
\begin{align}\label{S2}
 &\Opnorm{\Lambda^{-1}}\E\bigl[\HSnorm{S}\bigr]=\frac{n}{p_1}\E\bigl[\HSnorm{S}\bigr]\notag\\
 &\leq \frac{1}{p_1}\biggl(\sum_{i,k=1}^d\Var\Bigl(n\E\bigl[\bigl(W'(i)-W(i)\bigr)\bigl(W'(k)-W(k)\bigr)\,\bigl|\,X\bigr]\Bigr)\biggr)^{1/2}\notag\\
 &= \frac{1}{p_1}\biggl(\sum_{i,k=1}^d\sum_{q=0}^{p_i+p_k-1}\bigl(p_i+p_k-q\bigr)^2\sum_{\substack{M\subseteq[n]:\\\abs{M}=q}}\Var\bigl(U_M(i,k)\bigr)\biggr)^{1/2}\notag\\
 &\leq\frac{1}{p_1}\biggl(\sum_{i,k=1}^d\bigl(p_i+p_k\bigr)^2\sum_{\substack{M\subseteq[n]:\\\abs{M}\leq p_i+p_k-1}}\Var\bigl(U_M(i,k)\bigr)\biggr)^{1/2}\,.
\end{align}
Next, we turn to 
\begin{equation*}
 \Opnorm{\Lambda^{-1}}\E\bigl[\Enorm{W'-W}^3\bigr]=\frac{n}{p_1}\E\bigl[\Enorm{W'-W}^3\bigr]\,.
\end{equation*}
Using H\"older's inequality for sums and then the Cauchy-Schwarz inequality for $\E$, we obtain
\begin{align}\label{R1}
 \frac{n}{p_1}\E\bigl[\Enorm{W'-W}^3\bigr] &=\frac{n}{p_1}\Bigl(\sum_{i=1}^d\E\babs{W'(i)-W(i)}^2\cdot 1\Bigr)^{3/2}\notag\\
 &\leq\frac{n}{p_1}\Biggl(\biggl(\sum_{i=1}^d\E\babs{W'(i)-W(i)}^3\biggr)^{2/3}\cdot d^{1/3}\Biggr)^{3/2}\notag\\
 & =\frac{n}{p_1}d^{1/2}\sum_{i=1}^d\E\babs{W'(i)-W(i)}^3\notag\\
 &\leq\frac{n}{p_1}d^{1/2}\sum_{i=1}^d\biggl(\E\babs{W'(i)-W(i)}^2 \E\babs{W'(i)-W(i)}^4\biggr)^{1/2}\notag\\
 &=\frac{2\sqrt{2d}}{p_1}\sum_{i=1}^d\sigma_n(i)p_i\Bigl(\frac{n}{4p_i}\E\babs{W'(i)-W(i)}^4\Bigr)^{1/2}\,.
\end{align}
Here, we have used 
\begin{equation*}
 \E\babs{W'(i)-W(i)}^2=\frac{2p_i}{n}\E\bigl[W(i)^2\bigr]=\frac{2p_i}{n}\sigma_n(i)^2
\end{equation*}
to obtain the last identity. Now, taking into consideration that, in contrast to the one-dimensional setting, we did not normalize the components $W(i)$ of $W$, similarly to \eqref{1dle4} we obtain
\begin{align*}
 \frac{n}{4p_i}\E\babs{W'(i)-W(i)}^4&\leq 2\sum_{\substack{M\subseteq[n]:\\ \abs{M}\leq 2p_i-1}}\Bigl(1-\frac{\abs{M}}{2p_i}\Bigr)\Var\bigl(U_M(i,i)\bigr)+\kappa_{p_i}\frac{p_i}{n}\sigma_n(i)^4\,.
\end{align*}
Hence, from \eqref{R1} we conclude
\begin{align*}
 \Opnorm{\Lambda^{-1}}\E\bigl[\Enorm{W'-W}^3\bigr]&\leq \frac{4\sqrt{d}}{p_1}\sum_{i=1}^d\sigma_n(i)p_i\Biggl(\sum_{\substack{M\subseteq[n]:\\ \abs{M}\leq 2p_i-1}}\Bigl(1-\frac{\abs{M}}{2p_i}\Bigr)\Var\bigl(U_M(i,i)\bigr)\Biggr)^{1/2}\notag\\
 &\;+\frac{2\sqrt{2d}}{p_1\sqrt{n}}\sum_{i=1}^d\sigma_n(i)^3p_i^{3/2}\sqrt{\kappa_{p_i}}\notag\\
 &\leq   \frac{4\sqrt{d}}{p_1}\sum_{i=1}^d\sigma_n(i)p_i\Biggl(\sum_{\substack{M\subseteq[n]:\\ \abs{M}\leq 2p_i-1}}\Var\bigl(U_M(i,i)\bigr)\Biggr)^{1/2}\notag\\    
 &\;+\frac{2\sqrt{2d}}{p_1\sqrt{n}}\sum_{i=1}^d\sigma_n(i)^3p_i^{3/2}\sqrt{\kappa_{p_i}}        \,.
\end{align*}

\end{proof}

\begin{proof}[\bf Proof of Lemma \ref{genulemma}]
 Recall that we have 
\begin{align}\label{cb9}
 &\bigl(\psi_{i}\star_r^l\psi_{k}\bigr)(x_1,\dotsc,x_{k+i-2r},y_{l+1},\dotsc,y_r)\notag\\
 &=\int_{E^l}\psi_i(y_1,\dotsc,y_r,x_1,\dotsc,x_{i-r})\notag\\
 &\hspace{2cm}\cdot \psi_k(y_1,\dotsc,y_r,x_{i-r+1},\dotsc,x_{k+i-2r})d\mu^{\otimes l}(y_{1},\dots,y_l)\,.
\end{align}
Recalling also the expression \eqref{defpsis} of the kernels $\psi_i$ and $\psi_k$, respectively, and taking into account that $\mu$ is a probability measure as well as that, for $k\geq s$, the functions $g_k$ from \eqref{gk} satsify 
\begin{equation}\label{itg}
\int_{E^{k-s}}g_k(x_1,\dotsc,x_k)d\mu^{\otimes{k-s}}(x_{s+1},\dotsc,x_k)=g_s(x_1,\dotsc,x_s)
\end{equation}
by virtue of Fubini's theorem, we see that $ \bigl(\psi_{i}\star_r^l\psi_{k}\bigr)(x_1,\dotsc,x_{k+i-2r},y_{l+1},\dotsc,y_{r})$ is a linear combination, with coefficients only depending on $i,k,r$ and $l$ but not on $n$, of expressions of the form 
\begin{align*}
&G_{(a,b,j,m)}(x_{i_1},\dotsc,x_{i_{j-b-d}},y_{q_1},\dotsc, y_{q_c},x_{k_1},\dotsc, x_{k_{m-b-e}})\\
& :=\int_{E^t} g_j(u_1,\dotsc,u_b,y_{m_1},\dotsc,y_{m_d},  x_{i_1},\dotsc,x_{i_{j-b-d}}) \\
&\hspace{2cm}\cdot g_m(u_1,\dotsc,u_b,y_{n_1},\dotsc,y_{n_e},x_{k_1},\dotsc, x_{k_{m-b-e}})d\mu^{\otimes b}(u_1,\dots,u_b)
\\&=\bigl(g_j\star_a^b g_m\bigr)(x_{i_1},\dotsc,x_{i_{j-b-d}},y_{q_1},\dotsc, y_{q_c},x_{k_1},\dotsc, x_{k_{m-b-e}})
\,,
\end{align*}
where $0\leq j\leq i$, $0\leq m\leq k$, $0\leq b\leq l$, $0\leq b\leq a\leq r$, $1\leq i_1<\ldots<i_{j-b-d}\leq i-r$,\\
$i-r+1\leq k_1<\ldots<k_{l-b-e}\leq k+i-2r$ such that, in particular, the sets $\{i_1,\dotsc,i_{j-t-a}\}$ and $\{k_1,\dotsc,k_{m-t-b}\}$ are disjoint. 
Furthermore, we have $l+1\leq m_1<\ldots<m_d\leq r$, $l+1\leq n_1<\ldots<n_e\leq r$, $l+1\leq q_1<\ldots<q_c\leq r$ such that $\{q_1,\dotsc,q_c\}=\{m_1,\dotsc,m_d\}\cup\{n_1,\dotsc,n_e\}$, $d\leq j-b$, $e\leq m-b$ and 
$a:=b+|\{m_1,\dotsc,m_d\}\cap\{n_1,\dotsc,n_e\}|\leq b+c$. Note that $c\leq r-l$ and, hence, also $a-b\leq c\leq r-l$ as well as $a\leq(b+d)\wedge(b+e)\leq j\wedge m$.
Moreover, the number $j+m-a-b$ of arguments of the function $g_j\star_a^b g_m$ is at most as large as the number 
$i+k-r-l$ of arguments of the function $\psi_{i}\star_r^l\psi_{k}$. Finally, if $j=m=p$, then $i=k=p$ and $g_j=g_m=\psi$. This also implies that $b=l$ and $a=r$. Hence, we conclude that $(j,m,a,b)\in Q(i,k,r,l)$. 

Now, using the fact that $\mu$ is a probability measure, we obtain that 
\begin{align}\label{gl1}
 &\int_{E^{k+i-r-l}}G^2_{(a,b,j,m)}(x_{i_1},\dotsc,x_{i_{j-b-d}},y_{q_1},\dotsc, y_{q_c},x_{k_1},\dotsc, x_{k_{m-b-e}})\notag\\
 &\hspace{3cm} d\mu^{\otimes i+k-l-r}(x_1,\dotsc,x_{k+i-2r},y_{l+1},\dotsc,y_r)\notag\\
 &=\int_{E^{j+m-a-b}} \bigl(g_j\star_a^b g_m\bigr)^2d \mu^{\otimes j+m-a-b}=\|g_j\star_a^bg_m\|^2_{L^2(\mu^{\otimes j+m-a-b})}\,.
\end{align}
Since $\psi_{i}\star_r^l\psi_{k}$ is a finite linear combination with coefficients depending uniquely on $i,k,r$ and $l$ of the $G_{(a,b,j,m)}$, the claim thus follows from \eqref{gl1} and 
Minkowski's inequality.
\end{proof}

\section*{Acknowledgement} 
We would like to thank the anonymous referee for their useful comments and suggestions that helped us improve the presentation of our results.

\normalem
\bibliography{symU}{}

\begin{thebibliography}{NPR10b}

\bibitem[BG92]{BhaGo92}
R.~N. Bhattacharya and J.~K. Ghosh.
\newblock A class of {$U$}-statistics and asymptotic normality of the number of
  {$k$}-clusters.
\newblock {\em J. Multivariate Anal.}, 43(2):300--330, 1992.

\bibitem[BP16]{BPsv}
S.~Bourguin and G.~Peccati.
\newblock {The Malliavin-Stein method on the Poisson space}.
\newblock In G.~Peccati and M.~Reitzner, editors, {\em {Stochastic analysis for
  Poisson point processes}}, Mathematics, Statistics, Finance and Economics,
  chapter~6, pages 185--228. Bocconi University Press and Springer, 2016.

\bibitem[CM08]{ChaMe08}
S.~Chatterjee and E.~Meckes.
\newblock Multivariate normal approximation using exchangeable pairs.
\newblock {\em ALEA Lat. Am. J. Probab. Math. Stat.}, 4:257--283, 2008.

\bibitem[dJ89]{deJo89}
P.~de~Jong.
\newblock {\em Central limit theorems for generalized multilinear forms},
  volume~61 of {\em CWI Tract}.
\newblock Stichting Mathematisch Centrum, Centrum voor Wiskunde en Informatica,
  Amsterdam, 1989.

\bibitem[dJ90]{deJo90}
P.~de~Jong.
\newblock A central limit theorem for generalized multilinear forms.
\newblock {\em J. Multivariate Anal.}, 34(2):275--289, 1990.

\bibitem[DKP19]{DKP}
C.~D\"obler, M.~Kasprzak, and G.~Peccati.
\newblock {Functional convergence of $U$-processes with size-dependent
  kernels}.
\newblock Preprint, arXiv:1912.02705, 2019.

\bibitem[DM83]{DynMan83}
E.~B. Dynkin and A.~Mandelbaum.
\newblock Symmetric statistics, {P}oisson point processes, and multiple
  {W}iener integrals.
\newblock {\em Ann. Statist.}, 11(3):739--745, 1983.

\bibitem[D{\"o}b12]{Doe12c}
C.~D{\"o}bler.
\newblock {New developments in Stein's method with applications}.
\newblock 2012.
\newblock (Ph.D.)-Thesis Ruhr-Universit\"at Bochum.

\bibitem[DP17]{DP16}
C.~D\"obler and G.~Peccati.
\newblock {Quantiative de Jong theorems in any dimension}.
\newblock {\em Electron. J. Probab.}, 22:no. 2, 1--35, 2017.

\bibitem[Gre77]{Greg77}
G.~G. Gregory.
\newblock Large sample theory for {$U$}-statistics and tests of fit.
\newblock {\em Ann. Statist.}, 5(1):110--123, 1977.

\bibitem[Hal84]{Hall84}
P.~Hall.
\newblock Central limit theorem for integrated square error of multivariate
  nonparametric density estimators.
\newblock {\em J. Multivariate Anal.}, 14(1):1--16, 1984.

\bibitem[Hoe48]{Hoeffding}
W.~Hoeffding.
\newblock A class of statistics with asymptotically normal distribution.
\newblock {\em Ann. Math. Statistics}, 19:293--325, 1948.

\bibitem[JJ86]{JJ}
S.~R. Jammalamadaka and S.~Janson.
\newblock Limit theorems for a triangular scheme of {$U$}-statistics with
  applications to inter-point distances.
\newblock {\em Ann. Probab.}, 14(4):1347--1358, 1986.

\bibitem[KB94]{KB94}
V.~S. Koroljuk and Yu.~V. Borovskich.
\newblock {\em Theory of {$U$}-statistics}, volume 273 of {\em Mathematics and
  its Applications}.
\newblock Kluwer Academic Publishers Group, Dordrecht, 1994.
\newblock Translated from the 1989 Russian original by P. V. Malyshev and D. V.
  Malyshev and revised by the authors.

\bibitem[Kro17]{Krok15}
K.~Krokowski.
\newblock Poisson approximation of {R}ademacher functionals by the
  {C}hen-{S}tein method and {M}alliavin calculus.
\newblock {\em Commun. Stoch. Anal.}, 11(2):195--222, 2017.

\bibitem[KRT16]{KRT1}
K.~Krokowski, A.~Reichenbachs, and C.~Th\"ale.
\newblock Berry-{E}sseen bounds and multivariate limit theorems for functionals
  of {R}ademacher sequences.
\newblock {\em Ann. Inst. Henri Poincar\'e Probab. Stat.}, 52(2):763--803,
  2016.

\bibitem[Las16]{Lastsv}
G.~Last.
\newblock {Stochastic analysis for Poisson processes}.
\newblock In G.~Peccati and M.~Reitzner, editors, {\em {Stochastic analysis for
  Poisson point processes}}, Mathematics, Statistics, Finance and Economics,
  chapter~1, pages 1--36. Bocconi University Press and Springer, 2016.

\bibitem[LRP]{LRP3}
R.~Lachi{\`e}ze-Rey and G.~Peccati.
\newblock {New Kolmogorov bounds for functionals of binomial point processes}.
\newblock {\em to appear in: Ann. Appl. Probab.}

\bibitem[LRP13a]{LRP1}
R.~Lachi{{\`e}}ze-Rey and G.~Peccati.
\newblock Fine {G}aussian fluctuations on the {P}oisson space, {I}:
  contractions, cumulants and geometric random graphs.
\newblock {\em Electron. J. Probab.}, 18:no. 32, 32, 2013.

\bibitem[LRP13b]{LRP2}
R.~Lachi{{\`e}}ze-Rey and G.~Peccati.
\newblock Fine {G}aussian fluctuations on the {P}oisson space {II}: rescaled
  kernels, marked processes and geometric {$U$}-statistics.
\newblock {\em Stochastic Process. Appl.}, 123(12):4186--4218, 2013.

\bibitem[Maj13]{Maj13}
P.~Major.
\newblock {\em On the estimation of multiple random integrals and
  {$U$}-statistics}, volume 2079 of {\em Lecture Notes in Mathematics}.
\newblock Springer, Heidelberg, 2013.

\bibitem[Mec09]{Meck09}
E.~Meckes.
\newblock On {S}tein's method for multivariate normal approximation.
\newblock In {\em High dimensional probability {V}: the {L}uminy volume},
  volume~5 of {\em Inst. Math. Stat. Collect.}, pages 153--178. Inst. Math.
  Statist., Beachwood, OH, 2009.

\bibitem[NP12]{NouPecbook}
I.~Nourdin and G.~Peccati.
\newblock {\em Normal approximations with {M}alliavin calculus}, volume 192 of
  {\em Cambridge Tracts in Mathematics}.
\newblock Cambridge University Press, Cambridge, 2012.
\newblock From Stein's method to universality.

\bibitem[NPR10a]{NPR10}
I.~Nourdin, G.~Peccati, and G.~Reinert.
\newblock Invariance principles for homogeneous sums: universality of
  {G}aussian {W}iener chaos.
\newblock {\em Ann. Probab.}, 38(5):1947--1985, 2010.

\bibitem[NPR10b]{NPR10b}
I.~Nourdin, G.~Peccati, and G.~Reinert.
\newblock Stein's method and stochastic analysis of {R}ademacher functionals.
\newblock {\em Electron. J. Probab.}, 15:no. 55, 1703--1742, 2010.

\bibitem[Pen03]{Penrose}
M.~Penrose.
\newblock {\em Random geometric graphs}, volume~5 of {\em Oxford Studies in
  Probability}.
\newblock Oxford University Press, Oxford, 2003.

\bibitem[Pen04]{Penrose-book}
M.~Penrose.
\newblock {\em Geometric Random Graphs}.
\newblock Oxford, 2004.

\bibitem[PR16]{PecRei16}
G.~Peccati and M.~Reitzner.
\newblock {\em {Stochastic Analysis for Poisson Point Processes}}.
\newblock Mathematics, Statistics, Finance and Economics. Bocconi University
  Press and Springer, 2016.

\bibitem[PSTU10]{PSTU}
G.~Peccati, J.~L. Sol{{\'e}}, M.~S. Taqqu, and F.~Utzet.
\newblock Stein's method and normal approximation of {P}oisson functionals.
\newblock {\em Ann. Probab.}, 38(2):443--478, 2010.

\bibitem[PT15]{PrTo}
N.~Privault and G.~L. Torrisi.
\newblock The {S}tein and {C}hen-{S}tein methods for functionals of
  non-symmetric {B}ernoulli processes.
\newblock {\em ALEA Lat. Am. J. Probab. Math. Stat.}, 12(1):309--356, 2015.

\bibitem[PZ10]{PZ1}
G.~Peccati and C.~Zheng.
\newblock Multi-dimensional {G}aussian fluctuations on the {P}oisson space.
\newblock {\em Electron. J. Probab.}, 15:no. 48, 1487--1527, 2010.

\bibitem[RR97]{RiRo97}
Y.~Rinott and V.~Rotar.
\newblock On coupling constructions and rates in the {CLT} for dependent
  summands with applications to the antivoter model and weighted
  {$U$}-statistics.
\newblock {\em Ann. Appl. Probab.}, 7(4):1080--1105, 1997.

\bibitem[RR10]{ReiRoe10}
G.~Reinert and A.~R\"ollin.
\newblock Random subgraph counts and {$U$}-statistics: multivariate normal
  approximation via exchangeable pairs and embedding.
\newblock {\em J. Appl. Probab.}, 47(2):378--393, 2010.

\bibitem[Ser80]{Serf80}
R.~J. Serfling.
\newblock {\em Approximation theorems of mathematical statistics}.
\newblock John Wiley \& Sons, Inc., New York, 1980.
\newblock Wiley Series in Probability and Mathematical Statistics.

\bibitem[Ste86]{St86}
C.~Stein.
\newblock {\em Approximate computation of expectations}.
\newblock Institute of Mathematical Statistics Lecture Notes---Monograph
  Series, 7. Institute of Mathematical Statistics, Hayward, CA, 1986.

\bibitem[Sur84]{Sur84}
D.~Surgailis.
\newblock On multiple {P}oisson stochastic integrals and associated {M}arkov
  semigroups.
\newblock {\em Probab. Math. Statist.}, 3(2):217--239, 1984.

\bibitem[Vit92]{Vit92}
R.~A. Vitale.
\newblock Covariances of symmetric statistics.
\newblock {\em J. Multivariate Anal.}, 41(1):14--26, 1992.

\bibitem[Web83]{Weber}
N.~C. Weber.
\newblock Central limit theorems for a class of symmetric statistics.
\newblock {\em Math. Proc. Cambridge Philos. Soc.}, 94(2):307--313, 1983.

\end{thebibliography}
\bibliographystyle{alpha}
\end{document}